\documentclass[11pt]{amsart} 
\usepackage{amsmath, amsxtra, amsthm, amsfonts, amssymb,color}
\usepackage{tikz}
\usetikzlibrary{calc,arrows.meta,decorations.pathreplacing,positioning}

\usepackage{graphicx}
\usepackage[hidelinks,backref=page]{hyperref} 
\hypersetup{
	colorlinks,
	citecolor=magenta,
	filecolor=magenta,
	linkcolor=blue,
	urlcolor=black
}
\usepackage[top=2cm, bottom=2cm, right=2cm, left=2cm]{geometry}

\usepackage[capitalize]{cleveref}

\def\P{{\mathbb P}}
\def\R{{\mathbb R}}
\def\C{{\mathbb C}}
\def\Z{{\mathbb Z}}
\def\T{{\mathbb{T}}}
\def\B{{\mathcal{B}}}
\def\F{{\mathcal{F}}}

\def\r{{\mathbf{r}}}
\def\x{{\mathbf{x}}}
\def\y{{\mathbf{y}}}
\def\v{{\mathbf{v}}}

\def\h{{\mathfrak{h}}}

\def\conv{{\rm conv}}

\def\Trop{{\rm Trop}}
\def\Gr{{\rm Gr}}

\def\GL{{\rm GL}}
\def\sp{{\rm span}}

\def\tvn{\tilde \vn}
\def\vn{\mathfrak{t}}
\def\Tkn{\T^{k,n}}
\def\Rkn{\tilde{\T}^{k,n}}
\def\Lkn{L_{k,n}}
\def\wt{{\rm wt}}

\def\D{D}

\def\cR{{\mathcal{R}}}
\def\val{{\rm val}}
\def\t{{\mathfrak{t}}}
\def\hp{{\hat \tropP}}
\newcommand{\tropP}{{\pi}}  

\def\corank{{\rm corank}}

\def\PK{{\rm PK}}
\def\ncyc{{\rm ncyc}}
\def\WS{{\rm WS}}
\def\NC{{\rm NC}}
\def\wt{{\rm wt}}

\def\cJ{{\mathcal{J}}}

\def\bncyc{{\binom{[n]}{k}^\text{ncyc}}}

\def\cK{{\mathcal{K}}}

\newcommand{\one}{\mathbf{1}}

\newtheorem{theorem}{Theorem}
\newtheorem{lemma}[theorem]{Lemma}
\newtheorem{proposition}[theorem]{Proposition}
\newtheorem{thm}[theorem]{Theorem}
\newtheorem{cor}[theorem]{Corollary}
\newtheorem{corollary}[theorem]{Corollary}
\newtheorem{lem}[theorem]{Lemma}
\newtheorem{example}[theorem]{Example}
\newtheorem{definition}[theorem]{Definition}
\newtheorem{defn}[theorem]{Definition}

\newtheorem{conjecture}[theorem]{Conjecture}

\newtheorem{rem}[theorem]{Remark}
\newtheorem{remark}[theorem]{Remark}
\newtheorem{prop}[theorem]{Proposition}
\newtheorem{hypothesis}[theorem]{Hypothesis}
\newenvironment{nouppercase}{%
	\renewcommand{\uppercasenonmath}[1]{}}{}

\numberwithin{theorem}{section}
\numberwithin{equation}{section}
\author{Nick Early}
\author{Thomas Lam}
\address{School of Natural Sciences, Institute for Advanced Study, 1 Einstein Dr. Princeton, NJ 08540, USA.}
\email{\href{mailto:earlnick@ias.edu}{earlnick@ias.edu}}
\address{Department of Mathematics, University of Michigan, 2074 East Hall, 530 Church Street, Ann Arbor, MI 48109-1043, USA}
\email{\href{mailto:tfylam@umich.edu}{tfylam@umich.edu}}
\begin{document}
	\title{Noncrossing Duality and the Geometry of Positive Tropical Linear Spaces}

	\begin{nouppercase}
		\maketitle
	\end{nouppercase}
	
	\begin{abstract}

While the positive Grassmannian is deeply understood through the rich combinatorics of plabic graphs and positroid cells, its tropical counterpart, the positive tropical Grassmannian $\Trop_{>0}\Gr(k,n)$, has lacked a comparable structural framework for general $k$.  Both the global face structure of $\Trop_{>0}\Gr(k,n)$ and the internal metric geometry of the tropical linear spaces it parametrizes have remained largely uncharted. This paper develops a systematic algebraic and polyhedral foundation that resolves this gap.  
	
	The engine of our framework is a fundamental tropical duality, analogous to the duality between cluster variables (or more precisely, their $u$-coordinates) and $\mathbf{g}$-vectors, pairing two families of objects introduced by the first author: the planar basis of tropical Pl\"ucker vectors and the planar cross-ratios on the positive configuration space.  We prove that this duality links the fan structure of the positive tropical Grassmannian to the noncrossing fan of Santos, Stump, and Welker, yielding a global bijection between integer points of $\Trop_{>0}\Gr(k,n)$ and noncrossing tableaux.  
	
	We then study how this discrete combinatorial data controls the continuous metric geometry of positive tropical linear spaces.  We realize the bounded complex of an integer positive tropical linear space as the subdifferential of a central roof function on the hypersimplex, and use this realization to embed it into a dilate of the fundamental alcoved simplex.  The dilation factor, and hence the geometric diameter of the complex, is governed by a single invariant, the planar kinematics ($\PK$) weight, which we show equals the number of columns in the associated noncrossing tableau.  
	
	The results of this work are applied in our parallel work~\cite{EL} on scaffolds for higher tropical Grassmannians. 
	\end{abstract}

	\tableofcontents

\section{Introduction}\label{sec:introduction}

The tropical Grassmannian $\Trop \, \Gr(k,n)$ is a polyhedral fan introduced by Speyer and Sturmfels \cite{SS} and the positive tropical Grassmannian $\Trop_{>0} \Gr(k,n)$ is a subfan introduced by Speyer and Williams \cite{SW05}.  The case $k=2$ is classical and well understood: the tropical Grassmannian $\Trop \, \Gr(2,n)$ is the space of phylogenetic trees, with the positive part $\Trop_{>0} \Gr(2,n)$ given by the planar trees.  Through the identification of $\Trop_{>0} \Gr(2,n)$ with the tropical moduli space $\Trop_{>0} M_{0,n}$, the $k=2$ theory connects to the geometry of genus zero curves, open string amplitudes and biadjoint scalar $\phi^3$ amplitudes, and the combinatorics of the associahedron.

For general $k > 2$, the positive tropical Grassmannian similarly connects to the geometry of the configuration spaces of $n$ points on $\P^{k-1}$ \cite{ALS}, and to the CEGM generalized biadjoint scalar amplitudes \cite{CEGM}.  There are deep connections to cluster algebras and cluster fans \cite{SW05}, to the symbol alphabet of $N=4$ SYM amplitudes \cite{ALS2, DFGK, HP}, and to tilings of the $m=2$ amplituhedron \cite{LPW}.  While the positive Grassmannian $\Gr_{>0}(k,n)$ admits a rich combinatorial theory via plabic graphs and positroid cells \cite{Pos06}, there is no analogue of this machinery on the tropical side.  The face structure of $\Trop_{>0} \Gr(k,n)$, the internal geometry of the tropical linear spaces it parametrizes, and the factors controlling their size have remained largely unexplored.

This paper develops a systematic framework for these questions, and uses it to give an explicit geometric description of positive tropical linear spaces.  The starting point are the \emph{planar basis} $\{\h_J\}$ and the \emph{planar cross-ratios} $\{u_J\}$ introduced by the first author \cite{E19,E22}.  The planar basis elements are rays of the positive tropical Grassmannian, while the planar cross-ratios form a dual basis of torus-invariant functions on the positive configuration space $X(k,n)$.  

\medskip
\noindent
{\bf Noncrossing combinatorics of the positive tropical Grassmannian.} Our main result compares positive configuration space $\Trop_{>0}X(k,n)$ with the \emph{noncrossing fan} of Santos--Stump--Welker \cite{SSW17}.  Noncrossing tableaux are a variant on semistandard Young tableaux first introduced in \cite{PKPS10}, and whose polyhedral geometry was systematically studied in \cite{SSW17}.  We show that the projection of the planar basis onto the rays of the noncrossing fan induces a bijection between $\Trop_{>0}X(k,n)$ and the noncrossing fan $\NC_{k,n}$.  In particular, we deduce a bijection between noncrossing tableaux and integer points of positive configuration space.  The key computation is the duality theorem between tropical planar cross-ratios $u_J^t$ and the rays $\vn_J$ of the noncrossing fan.  This duality also plays a key role in our companion paper \cite{EL}.

\begin{figure}[h!]
	\centering
	\includegraphics[width=1\linewidth]{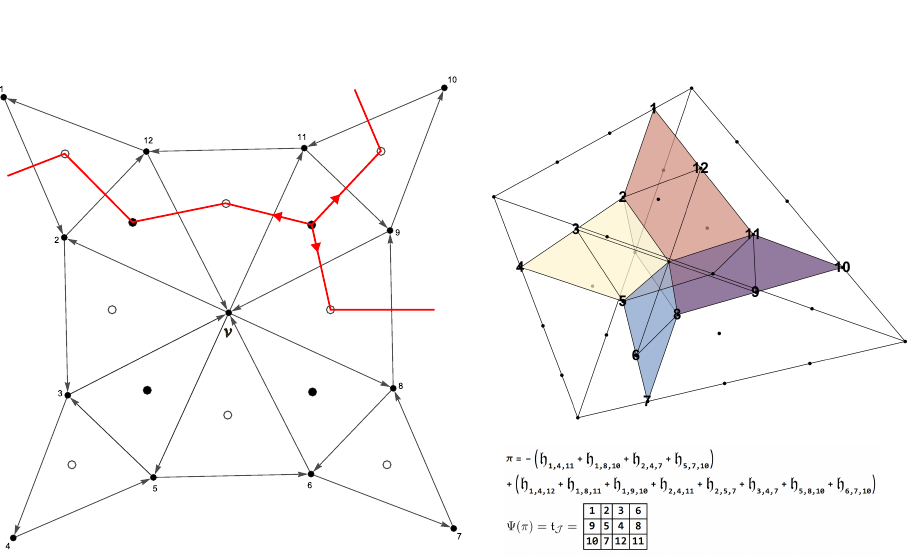}
	\caption{
		A scaffold (left), defined in our companion paper \cite{EL}, embeds isometrically in the tropical linear space, illustrated on the right by framing it with an embedded tetrahedron in $\mathbb{R}^3$ with all triangles equilateral.  The red strand triple $(1,8,10)$ (left) encodes the summand $-\h_{1,8,10}$ in the planar basis expansion on the bottom right. According to Theorem \ref{thm:global_bounding_simplex}, the bounded complex of the tropical linear space embeds in the fourth dilate of an $11$-dimensional simplex; and here, the edge length of the bounding tetrahedron is four, which is the number of columns in the noncrossing tableau representation of the positive tropical Pl\"ucker vector.  See Theorem \ref{thm:f-pk-equals-nc} and Theorem \ref{thm:global_bounding_simplex} for the relation between the number of columns and the diameter, and Example \ref{example: 312 embedding}, where the positive tropical Pl\"ucker vector illustrated here is examined.}
	\label{fig:312canvas}
\end{figure}

\medskip
\noindent
{\bf Geometry of positive tropical linear spaces.}
 Each positive tropical Pl\"ucker vector $\tropP_\bullet$ determines a tropical linear space $L(\tropP_\bullet)$, and we study the \emph{bounded subcomplex} $B(\tropP_\bullet)\subset L(\tropP_\bullet)$ consisting of the bounded faces.  Each face of $B(\tropP_\bullet)$ is a polypositroid, a special class of alcoved polytopes \cite{LP07,LP}.  For each $\tropP_\bullet \in \Trop_{>0}\Gr(k,n)$, we construct a \emph{central} piecewise-linear function $\hp(x)$ on the hypersimplex $\Delta(k,n)$ and show that the subdifferential of $\hp(x)$ recovers the bounded complex: $B(\tropP_\bullet) = \bigcup_{x \in \operatorname{int}(\Delta(k,n))} \widetilde\partial\hp(x)$ (Proposition~\ref{prop:bounded-subdiff}).  
 
Our results on noncrossing tableaux are applied to the computation of the diameter of the bounded complex.  We define an invariant $\wt_\PK \in \Z_{\geq 0}$, the \emph{weight}, that stratifies the positive tropical Grassmannian; roughly speaking, the higher the weight the more complicated the corresponding point.  We show that under the correspondence with noncrossing tableau $\cJ$, the weight agrees with the size (number of columns) of $\cJ$.  We then show that $B(\tropP_\bullet)$ is contained in the $\wt_P(\tropP_\bullet)$-fold dilate of  the fundamental alcoved simplex (Theorem~\ref{thm:global_bounding_simplex}).  In other words, the combinatorial invariant $\wt_\PK(\tropP_\bullet)$ controls the diameter of the tropical linear space $L(\tropP_\bullet)$.  This confirms our intuition that weight reflects complexity.

\subsection*{Structure of the paper}
In Section~\ref{sec: main results} we state the main results.  Section~\ref{sec: cross-ratios} collects background on planar cross-ratios and the planar basis.  Section~\ref{sec:noncrossing} develops the combinatorics of noncrossing tableaux and the noncrossing fan.  Sections~\ref{sec:dualityproof} and~\ref{sec:f-tropical} prove the duality theorem and the weight identity $\wt_\PK = \wt_\NC$.  Section~\ref{sec: weight strat} develops the weight stratification.  

Sections~\ref{sec:trop-linear-spaces} and~\ref{sec:global-bounds} studies tropical linear spaces and the diameter of the bounded complex. 	

\subsection*{Outlook}
In our parallel work \cite{EL}, we use the tools developed here to introduce scaffolds for higher tropical Grassmannians.  In joint work of the first author with Pierre-Guy Plamondon and Hugh Thomas, the noncrossing complex, the associated u-variables and u-equations are interpreted in the setting of the representation theory of gentle algebras \cite{EPT}.  
It would also be interesting to apply our results in the context of scattering amplitudes and quantum affine algebras \cite{DGL,ELi,ELi25}. 

\medskip
\noindent
{\bf Acknowledgements.}
N.E. was funded by the European Union (ERC, UNIVERSE PLUS, 101118787). \begin{tiny}
	Views and opinions expressed are however those of the author(s) only and do not necessarily reflect those of the European Union or the European Research Council Executive Agency. Neither the European Union nor the granting authority can be held responsible for them.
\end{tiny}
T.L. was supported by the National Science Foundation under Grant No. DMS-2348799.

	\section{Main results}\label{sec: main results}	
	Throughout the paper, we fix integers $2 \le k < n$.  We write $[n] := \{1, \ldots, n\}$ 
	and $\binom{[n]}{k}$ for the set of $k$-subsets of $[n]$.  Indices in $[n]$ 
	are taken modulo $n$ when working with cyclic structures.  A \emph{cyclic 
		interval} is a subset of the form $\{a, a+1, \ldots, a+m-1\} \pmod{n}$.  A subset 
	is \emph{non-cyclic} if it is not a cyclic interval.  We write 
	$\bncyc$ for the set of non-cyclic $k$-subsets; note that 
	$\left|\bncyc\right| = \binom{n}{k} - n$.  For $J \subseteq [n]$, 
	we set $e_J = \sum_{j \in J} e_j \in \R^n$, where $e_1, \ldots, e_n$ are the 
	standard basis vectors.  When there is no ambiguity, we write $k$-subsets 
	without braces or commas (e.g., $ijk$ for $\{i,j,k\}$).
	
	For $N \ge 1$, let 
	$\T^{N-1} := \R^N / \R(1, \ldots, 1)$ denote the quotient vector space.  We let $\Z^N \subset \R^N$ and $\T_{\Z}^{N-1}:= \Z^N / \Z(1, \ldots, 1)$ denote the sets of integer points.
	Define product spaces 
	$$\Rkn := (\R^{n-k})^k, \qquad \text{and} \qquad 
	\Tkn := (\T^{n-k-1})^{k-1},$$ which have integer points $\Rkn_\Z=(\Z^{n-k})^k$ and $\Tkn_\Z$ respectively.
	\subsection{Tropical Grassmannian}\label{ssec:tropGr}
	Let $\Gr(k,n)$ denote the Grassmannian of $k$-planes in $\C^n$, and let $\Gr^\circ(k,n) \subset \Gr(k,n)$ denote the subspace where all Pl\"ucker coordinates are non-vanishing.  The torus $T \subset \GL(n)$ acts on $\Gr(k,n)$ preserving $\Gr^\circ(k,n)$.  The quotient $X(k,n) := \Gr^\circ(k,n)/T$ is the \emph{configuration space} of $n$ points in $\P^{k-1}$ in general linear position.
	
	Let $\cR$ denote the ring of Puiseux series over $\C$, with valuation $\val: \cR \to \mathbb{Q} \cup \{\infty\}$.  The \emph{tropical Pl\"ucker vector} $\tropP_\bullet(V)$ of $V \in \Gr^\circ(k,n)(\cR)$ is given by
	$$
	\tropP_I(V) = \val(p_I(V)) \in \mathbb{Q}, \qquad I \in \binom{[n]}{k},
	$$
	where $p_I(V)$ denotes the $I$-th Pl\"ucker coordinate of $V$.  The \emph{tropical Grassmannian} $\Trop\,\Gr(k,n)$ is
	$$ 
	\Trop\,\Gr(k,n) := \overline{\bigl\{\tropP_\bullet(V) \mid V \in \Gr^\circ(k,n)(\cR)\bigr\}} \subset \R^{\binom{[n]}{k}},
	$$
	which we view as a polyhedral fan of dimension $k(n-k)+1$.  The \emph{positive tropical Grassmannian} $\Trop_{>0}\Gr(k,n) \subset \Trop\,\Gr(k,n)$ is the subfan consisting of those tropical Pl\"ucker vectors arising from points in $\Gr^\circ(k,n)(\cR_{>0})$, where $\cR_{>0}$ denotes the semifield of Puiseux series with positive leading coefficient.  It has an alternative description in terms of positive tropical Pl\"ucker relations.
	
	A vector $\tropP_\bullet \in \R^{\binom{[n]}{k}}$ is a \emph{positive tropical Pl\"ucker vector} if it satisfies the \emph{positive tropical Pl\"ucker relation}:
	\begin{equation}\label{eq:posTropPlucker}
	\tropP_{Sac} + \tropP_{Sbd} = \min\bigl(\tropP_{Sab} + \tropP_{Scd},\ \tropP_{Sad} + \tropP_{Sbc}\bigr),
	\end{equation}
	for every $S \in \binom{[n]}{k-2}$ and $a < b < c < d$ in $[n] \setminus S$.  
	
	\begin{theorem}[\cite{ALS,SW21}]
	The set of positive tropical Pl\"ucker vectors coincides with the positive tropical Grassmannian $\Trop_{>0}\Gr(k,n)$.
	\end{theorem}
	The \emph{tropical configuration space} $\Trop\,X(k,n)$ is the $(k-1)(n-k-1)$-dimensional quotient of $\Trop\,\Gr(k,n)$ by its lineality space \eqref{eq:lineality}.  The \emph{positive tropical configuration space} $\Trop_{>0}X(k,n)$ is the image of $\Trop_{>0}\Gr(k,n)$ in $\Trop\,X(k,n)$.  A point in $\Trop\,\Gr(k,n)$ is \emph{integral} if it belongs to $\Z^{\binom{[n]}{k}}$.  A point in $\Trop\,X(k,n)$ is integral if it admits an integral representative.
	
	For a rational function $f$ on $\Gr(k,n)$, we let $f^t$ denote its tropicalization on $\Trop\,\Gr(k,n)$, a piecewise-linear function.  The tropical Pl\"ucker coordinates are $p_I^t : \R^{\binom{[n]}{k}} \to \R$, and they are the coordinate linear functions.

	For a positive tropical Pl\"ucker vector $\tropP_\bullet$ and $\x \in \R^n$, define $\tropP^{\x}_\bullet$ by
	$$
	\tropP^{\x}_I := \tropP_I - \sum_{i \in I} x_i, \qquad I \in \binom{[n]}{k},
	$$
	which is again a positive tropical Pl\"ucker vector.  We write $\tropP_\bullet \equiv \tropP^{\x}_\bullet$ to denote tropical Pl\"ucker vectors that are equal modulo $\Lkn$.

	\subsubsection*{Positive parametrization}
	We will work with a distinguished positive parametrization (Definition~\ref{def:positive-param}) of positive configuration space
	\begin{equation}\label{eq:Mx}
	M(\x) : (\R_{>0}^{n-k}/\R_{>0})^{k-1} \xrightarrow{\,\cong\,} X_{>0}(k,n),
	\end{equation}
	a variation on Postnikov's parametrization of the positive Grassmannian \cite{Pos06}.  Tropicalizing gives a homeomorphism
	\begin{equation}\label{eq:rho}
	\rho : \Tkn \xrightarrow{\,\cong\,} \Trop_{>0}X(k,n),
	\end{equation}
	a variant of the parametrization of \cite{SW05}.  Different tropical positive parametrizations of $\Trop_{>0}\Gr(k,n)$ are related by piecewise linear isomorphisms; the choice $M(\x)$ is distinguished by its relation to the noncrossing fan, stated in \cref{thm:noncrossingmain} below.

	\subsection{Noncrossing fan}
	
	Following \cite{CE24,E21}, for $J = \{j_1 < j_2 < \cdots < j_k\} \in \binom{[n]}{k}$, define the integer points 
	$$
	\vn_J = \sum_{i=1}^{k-1} \sum_{s = j_i - (i-1)}^{j_{i+1} - (i+1)} e_{i,s} \in \Tkn_\Z.
	$$
	
	\begin{defn}[Weak separation and noncrossing]\label{defn:weak-sep-nc}
		A pair $I, J \in \binom{[n]}{k}$ is \emph{weakly separated} if $e_I - e_J$ has at 
		most two sign changes when read cyclically \cite{LZ98}.  The pair is 
		\emph{noncrossing} (with respect to the linear order $1 < \cdots < n$) if 
		for each $1 \le a < b \le k$, writing $I = \{i_1 < \cdots < i_k\}$ and 
		$J = \{j_1 < \cdots < j_k\}$, either
		\begin{enumerate}
			\item the subpair $(\{i_a, \ldots, i_b\}, \{j_a, \ldots, j_b\})$ is 
			weakly separated, or 
			\item the interiors $\{i_{a+1}, \ldots, i_{b-1}\}$ and 
			$\{j_{a+1}, \ldots, j_{b-1}\}$ differ.
		\end{enumerate}
		The pair $(I, J)$ is \emph{crossing} if it is not noncrossing.\end{defn}
	
	Every weakly separated pair is noncrossing, but not conversely.  Denote by $\NC_{k,n}$ the simplicial complex of all collections of pairwise 
	noncrossing \textit{noncyclic} $k$-element subsets, ordered by inclusion.  This is a pure simplicial complex of dimension $(k-1)(n-k-1)-1$ \cite{PKPS10,SSW17}; see Section~\ref{sec:noncrossing}.  For each face $\cK \in \NC_{k,n}$, we define $C_{\cK}$ to be the cone in $\Tkn_\Z$ generated by the vectors $\{\vn_J \mid J \in \cK\}$.  If $\cK$ is the empty set, then $C_\emptyset = \{0\}$.
	
	\begin{rem}
		The complex $\WS_{k,n}$ of weakly separated collections is a simplicial subcomplex of $\NC_{k,n}$, and it is the intersection of the $n$ cyclic relabelings of $\NC_{k,n}$, see \cite{SSW17}.
	\end{rem}

	\begin{defn}
		The \emph{noncrossing fan} $F_{\NC_{k,n}}$ is the collection of cones $C_\cK$, as $\cK$ varies over the faces of $\NC_{k,n}$.
	\end{defn}
	
	The following result is deduced from the results of \cite{SSW17}; see Section~\ref{sec:noncrossing}.
	\begin{theorem}\label{thm:NCfan}
		The noncrossing fan $F_{\NC_{k,n}}$ is a complete unimodular fan in $\Tkn$.
	\end{theorem}

	A \emph{noncrossing tableau} $\cJ$ is a multiset $\cJ = \{J_1,J_2,\ldots,J_r\}$ of elements in $\bncyc$ such that every pair in $\cJ$ is noncrossing.  Given a noncrossing tableau $\cJ$, we define the point $\vn_{\cJ} \in \Tkn$ by
	$$
	\vn_{\cJ}:= \sum_{J \in \cJ} \vn_J \in \Tkn_\Z.
	$$

				\begin{example}
		The 14 maximal noncrossing collections in $\NC_{2,6}$ modulo cyclic relabeling are:
		$$(13,15,35),\ (13,14,15),\ (14,24,46).$$
		The noncrossing fan $F_{\NC_{2,6}}$ is the normal fan to the associahedron, with f-vector $(14,21,9)$.
		
		On the other hand, among the 42 maximal noncrossing collections in $\NC_{3,6}$, there are 34 maximal weakly separated collections and eight noncrossing collections which are not weakly separated, containing either $(124,356)$ or $(145,236)$.  Representatives from each orbit of the maximal weakly separated collections are given by 
		$$\begin{array}{cccc}
			(124,125,134,145) & (125,134,135,145) & (134,135,136,145) & (124,134,145,146) \\
			(134,136,145,146) & (135,136,145,235) & (124,125,145,245) & (146,236,245,246) \\
		\end{array}
		$$
		The eight additional noncrossing but not weakly separated collections are 
		$$\begin{array}{cccc}
			(145,146,236,245)& (136,145,146,236)&(136,145,235,236)&(145,235,236,245) \\
			(124,256,346,356)&(124,134,346,356)&(124,125,256,356)&(124,125,134,356) \\
		\end{array}
		$$
		This accounts for the 42 maximal cones in the noncrossing fan $F_{\NC_{3,6}}$. 
		\end{example}
	We show that the combinatorics of noncrossing tableaux interacts well with the fan structure of the positive tropical Grassmannian.

	\subsection{Planar basis}

	Let $\Delta(k,n)$ denote the hypersimplex with vertices $\{e_I \in \R^{n} \mid I \in \binom{[n]}{k}\}$, where $e_I:= e_{i_1} + \cdots + e_{i_k}$.  Let the \emph{directed distance} $d(e_I, e_J)$ on $\Delta(k,n)$ be the minimal number of steps of the form $e_i - e_{i+1}$ (indices mod $n$)  required to travel from $e_I$ to $e_J$ along edges of $\Delta(k,n)$.
	
	\begin{defn}[Planar basis \cite{E22}]\label{def:planar-basis}
		For $J \in \binom{[n]}{k}$, the  
		\emph{planar basis element} $\h_J \in \R^{\binom{[n]}{k}}$ is
		\[
		\h_J := \frac{1}{n} \sum_{I \in \binom{[n]}{k}} d(e_J, e_I) \, e^I.
		\]
	\end{defn}
	For example, we have
	$$
	\h_{13} = \frac{1}{4} \left( e^{12}+ 3e^{14} + 3e^{23}+2e^{24}+e^{34} \right).
	$$

	The elements $\{\h_J\}$ form a basis of $\R^{\binom{[n]}{k}}$, and the elements $\{\h_J : J \text{ cyclic}\}$ 
	span the lineality space $\Lkn$; see Theorem~\ref{thm:planarbasis}.  We define a map 
	\begin{align*}
		\Psi: \R^{\binom{[n]}{k}} &\longrightarrow \Tkn \\
		\h_J &\longmapsto \vn_J.
	\end{align*}

	\begin{theorem}\label{thm:noncrossingmain}
		The linear map $\Psi$ descends to $\R^{\binom{[n]}{k}}/L_{k,n}$ and restricts to a bijection $\Psi: \Trop_{>0} X(k,n) \cong \Tkn$, inverse to $\rho$ in \eqref{eq:rho}.   It induces a bijection
		$$
		\Trop_{>0}(X(k,n))(\Z) \,\longleftrightarrow\, \{\text{noncrossing tableaux}\}
		$$
		given by $\tropP_\bullet \mapsto \cJ$ if $\Psi(\tropP_\bullet) = \vn_{\cJ}$.  
	\end{theorem}
	
	The bijection of Theorem~\ref{thm:noncrossingmain} provides a purely combinatorial parametrization of integer points 
	in the positive tropical Grassmannian.
	
	\subsection{Planar cross-ratios}
	For 
	$J = \{j_1 < \cdots < j_k\} \in \binom{[n]}{k}$, define the set $I_J$ of 
	\emph{cyclic endpoints} of $J$ by
	\[
	I_J = \bigl\{ j \in J : j+1 \notin J \bigr\} 
	= \bigl\{ j \in [n] : (j, j+1) \in J \times J^c \bigr\},
	\]
	where indices are taken modulo $n$.
	
	\begin{definition}\label{def:cubical-array}
		For $J \in \binom{[n]}{k}$, the \emph{cubical array} $C_J$ is the 
		collection of $k$-subsets
		\[
		C_J = \left\{ e_J + \sum_{m \in M} (-e_m + e_{m+1}) : M \subseteq I_J \right\},
		\]
		where $e_J = \sum_{j \in J} e_j$ denotes the indicator vector of $J$, and 
		$M = (m_1, m_2, \ldots)$.
	\end{definition}
	
	By a \emph{cross-ratio} we mean a Laurent monomial in Pl\"ucker coordinates that is torus-invariant, that is, descends to the configuration space $X(k,n)$.
	
	\begin{definition}[Planar cross-ratios \cite{E19}]\label{def:u-variable}
		For $J \in \binom{[n]}{k}$, the \emph{planar cross-ratio} $u_J$ is defined by
		\[
		u_J = \prod_{e_M \in C_J} p_M^{(-1)^{|J \cap M| - k - 1}},
		\]
		where the product is over all $k$-subsets $M$ in the cubical array $C_J$, 
		and $p_M$ denotes the Pl\"ucker coordinate indexed by $M$.
	\end{definition}
	Note that $u_J$ is a rational function on $\Gr(k,n)$, and its tropicalization $u_J^t$ is a \emph{linear} function on $\R^{\binom{[n]}{k}}$.
	
	We may view $u_J$ as a rational function on $(\R_{>0}^{n-k}/\R_{>0})^{k-1}$ by pulling back along the positive parametrization \eqref{eq:Mx}.  Tropicalizing, we obtain the tropical planar cross-ratio $u_J^t$, a piecewise-linear function on $\Tkn$.
	\begin{proposition}
		The tropical positive parametrization $\rho: \Tkn \to \Trop_{>0}X(k,n)$ can be given explicitly as
		\begin{equation}\label{eq:rho}
		\rho(\vn) = \sum_{J \in \bncyc} u^t_J(\vn)\h_J.
		\end{equation}
	\end{proposition}
	\begin{proof}
	The equality $\tropP_\bullet \equiv \sum_{J \in \bncyc} u^t_J(\tropP_\bullet) \h_J$ holds for any $\tropP_\bullet$.  By definition, $u^t_J(\vn) = u^t_J(\rho(\vn))$.  
	\end{proof}
	
	\begin{remark}
	\cref{thm:duality} should be compared with \cite[Theorem 6.6]{AHL} which states a duality between tropical cross-ratios of cluster variables (called $u$-coordinates) and $\mathbf{g}$-vectors, for any cluster algebra of finite type.
	\end{remark}
	
	\begin{remark}
		The formula \eqref{eq:rho} is denoted $\F$ and called the Global Schwinger Parameterization in \cite{CE24}, where it is used to compute generalized biadjoint amplitudes.  In the same paper, the PK weight, introduced in the next section, was discovered while studying critical points of mirror superpotentials.
	\end{remark}
	
	\begin{theorem}\label{thm:duality}
		For all $k \geq 2$, $n \geq k+2$, and all $J, J' \in \bncyc$, we have
		$$u^t_J(\vn_{J'}) = \delta_{J,J'}.$$
	\end{theorem}
	
	\subsection{Weight filtration of tropical Grassmannian}
	We study two notions of weight, PK weight and NC weight.  The weight function stratifies the tropical Grassmannian: higher weight typically means more complicated.

	\begin{definition}[PK Weight]\label{defn:f-pk-weight}
		Define the \emph{PK (planar kinematics) weight} as the linear map $\wt_{\PK}: \R^{\binom{[n]}{k}} \to \R$ given by
		$$\wt_\PK(\h_J)  = \begin{cases}
			1, & J \in \bncyc \\
			0, & J \text{ cyclic}.
		\end{cases}$$
	\end{definition}

	Since the $\h_J$ with $J$ cyclic span $L_{k,n}$, the function $\wt_\PK$ descends to the quotient $\R^{\binom{[n]}{k}}/L_{k,n}$.
	
	\begin{definition}[NC Weight]\label{defn:f-extended-nc}
		For $\vn \in \Tkn$, write
		$$\vn = \sum_{j} \mu_j \vn_{K_j}$$
		where $\mu_j \geq 0$ and $\{K_j\}$ is a face of the noncrossing complex. Define the NC weight (noncrossing weight) by
		$$\wt_{\NC}(\vn) := \sum_j \mu_j.$$
	\end{definition}
	
	Our main result on weights is that PK weight and NC weight agree.
	\begin{theorem}[PK Weight Equals NC Weight]\label{thm:f-pk-equals-nc}
		For any $\vn \in \Tkn$,
		$$\wt_{\PK}(\rho(\vn)) = \wt_{\NC}(\vn).$$
	\end{theorem}
	
	\begin{example}
		Suppose that $(k,n) = (2,n)$.  All rays of $\Trop_{>0} X(2,n)$ have PK weight one, since
		$$\wt_{\PK}(\rho(\t_{i,j})) = \wt_{\PK}(\h_{i,j})= 1$$
		and these are known to be all the rays.
	\end{example}
	\begin{example}	
		Suppose that $(k,n) = (3,6)$.  The two rays $\tropP_\bullet = -\h_{135} + \h_{235} + \h_{145} + \h_{136}$ and $\tropP'_\bullet = -\h_{246} + \h_{124} + \h_{256} + \h_{346}$ have PK weight two.  They also have NC weight two, since
		$$\Psi(\tropP_\bullet) = \t_{145} + \t_{236}, \qquad \text{and} \qquad \Psi(\tropP'_\bullet) = \t_{124} + \t_{356},$$
		where by Theorem \ref{thm:noncrossingmain} the projection $\Psi$ is inverse to the positive parameterization $\rho$. 
	\end{example}
	
	We say that a ray $\mathbf{r}$ of $\Trop_{>0} X(k,n)$ has weight $r$ if the primitive integer generator $\v$ of the ray has weight $\wt_{\PK}(\v) = r$.
	We use Theorem~\ref{thm:f-pk-equals-nc} to explore the structure of the low weight part of the positive tropical Grassmannian.  By Theorem~\ref{thm:weak-separation-blade-arrangement}, for non-cyclic $J$, the ray 
	$\R_{\ge 0} \cdot \h_J$ is a ray of $\Trop_{>0} X(k,n)$ \cite{E22}.  The weight-one classification follows from Theorem~\ref{thm:f-pk-equals-nc}; see \cref{prop: PK weight general}.
	
	\begin{theorem}[Weight-One Classification]
		The rays $\R_{\ge 0} \cdot \h_J$ are exactly the weight one rays of $\Trop_{>0} X(k,n)$.
	\end{theorem}
	
	The following \Cref{conj:weighttwo} was made in \cite{E22B}.
	
	\begin{conjecture}[Weight-Two Classification]\label{conj:weighttwo}
		Suppose that $I, J \in \bncyc$ are noncrossing but not weakly separated. Then $\rho(\vn_I + \vn_J)$ lies in the direction of a ray of $\Trop_{>0}X(k,n)$.  	All rays of $\Trop_{>0} X(k,n)$ of PK weight $2$ are of this form for a unique noncrossing but not weakly separated pair $\{I, J\}$.
	\end{conjecture}
	
	In \cref{sec: weight strat}, we show that \cref{conj:weighttwo} holds modulo an integrality assumption.

\subsection{Diameter of the bounded complex}

Let $\Delta(k,n) := \conv\bigl(e_I : I \in \binom{[n]}{k}\bigr)$ denote the hypersimplex, where $e_I = \sum_{i \in I} e_i$.  A positive tropical Pl\"ucker vector $\tropP_\bullet$, viewed as a height function on the vertices of $\Delta(k,n)$, induces a regular subdivision $\Delta(\tropP_\bullet)$ of $\Delta(k,n)$ into positroid polytopes \cite{Spe,ALS,SW21}.  For each $\x \in \R^n$, we have a matroid $M(\tropP^\x_\bullet)$, and the matroid polytopes $P_{M(\tropP^\x_\bullet)}$, as $\x$ varies over $\R^n$ give the subdivision $\Delta(\tropP_\bullet)$.
	
	Let $\one := (1, 1, \ldots, 1) \in \Z^n$.
	\begin{definition}\label{def:tropical}
	The \emph{(positive) tropical linear space} of a (positive) tropical Pl\"ucker vector $\tropP_\bullet$ is
		\begin{equation}\label{eq:tropical-linear-space}
	L(\tropP_\bullet) := \Bigl\{ \x \in \R^n/\one : \text{for every } \tau \in \tbinom{[n]}{k+1},\ \min_{i \in \tau}\bigl(\tropP_{\tau \setminus \{i\}} + x_i\bigr) \text{ is achieved at least twice}\Bigr\},
	\end{equation}
	a subspace of $\R^n/\one$.
	\end{definition}
	We have that $\x \in L(\tropP_\bullet)$ if and only if $M(\tropP^\x_\bullet)$ is loopless (\cref{prop:Spe}).  The bounded subcomplex $B(\tropP_\bullet) \subset L(\tropP_\bullet)$ consists of those faces of $L(\tropP_\bullet)$ dual to matroid polytopes in the subdivision $\Delta(\tropP_\bullet)$ that are in the interior of $\Delta(k,n)$. 
	
	The following result states that $\wt_\PK(\tropP_\bullet)$ is a bound on the diameter of $B(\tropP_\bullet)$. 

\begin{theorem}\label{thm:global_bounding_simplex1}
	For any $\tropP_\bullet \in \Trop_{>0}\Gr(k,n)$, let $\hp_\bullet$ denote the \emph{balanced central representative} of \cref{lem:balanced-central}.  Then the bounded complex $B(\hp_\bullet)$ is contained in the $\wt_{\PK}(\tropP_\bullet)$-fold dilate of the standard alcoved simplex (\cref{def:std-simplex}).
\end{theorem}
\cref{thm:global_bounding_simplex1} is made more precise in \cref{thm:global_bounding_simplex}.

	\subsection{Associahedron case}		
	Let $k = 2$.  The pairs $ij \in \binom{[n]}{2}$ can be identified with the sides and diagonals of an $n$-gon.  A pair $ij$ is cyclic if it corresponds to a side and non-cyclic if it corresponds to a diagonal.  The pairs $ab$ and $cd$ are weakly separated if and only if they are noncrossing if and only if the corresponding sides/diagonals do not intersect in the interior of the $n$-gon.  Thus the noncrossing complex $\NC_{2,n}$ can be identified with the simplicial complex of subdivisions of an $n$-gon.  This is exactly the dual of the face poset of the associahedron polytope.  A noncrossing tableau in this case is a multiset of noncrossing diagonals of the $n$-gon.
	
	The set $\{\h_{ij} \mid (ij) \text{ diagonal }\}$ of planar basis elements is exactly the set of rays of $\Trop_{>0} X(2,n)$, and the map of \cref{thm:noncrossingmain} identifies the fan structure of $\Trop_{>0} X(2,n)$ with the noncrossing fan $F_{\NC_{2,n}}$, and is not just a bijection.  Modulo lineality, we have
	\begin{equation}\label{eq:hij}
	\h_{ij} \equiv \sum_{i < a < b \leq j} e^{ab} \equiv \sum_{j < a < b \leq i} e^{ab}.
	\end{equation}
	For example, 
	$$
	\h_{26} \equiv e^{34}+e^{35}+e^{36}+e^{45}+e^{46}+e^{56} \equiv e^{12}.
	$$
	The planar cross-ratios 
	$$
	u_{ij} = \frac{p_{i+1,j} p_{i,j+1}}{p_{ij} p_{i+1,j+1}}
	$$
	are (usual) cross ratios on $M_{0,n}$, called the \emph{dihedral coordinates}; see for example \cite{AHL,LamMod}.  It is straightforward to check using \eqref{eq:hij} that we have $u_{ij}^t(\h_{kl}) = \delta_{ij,kl}$, for two diagonals $ij$ and $kl$.  The result \cref{thm:f-pk-equals-nc} is trivial since the planar basis expansion and the noncrossing fan expansion coincide.

	Every point $\tropP_\bullet \in \Trop_{>0} X(2,n)$ has a \emph{nonnegative} expansion in the planar basis $\{\h_{ij} \mid (ij) \text{ diagonal }\}$.  In this case, the weight of $\tropP_\bullet$ is simply the sum of the coefficients in this expansion.  Suppose that $\tropP_\bullet$ corresponds to the multiset of pairwise noncrossing diagonals $\{(a_1,b_1),(a_2,b_2),\ldots,(a_r,b_r)\}$.  Then the bounded complex $B(\tropP_\bullet)$ is a tree $T = T(\tropP_\bullet)$ with $r$ edges, embedded into $\R^n/\one$.  The set $[n]$ labels the leaves of the tree $T$, and an edge $e$ corresponds to a diagonal $(a,b)$ if it separates labels $[a+1,b]$ from $[b+1,a]$.  In this case, the edge $e$ is a line segment that is a translate of $[0, e_{[a+1,b]}]$.  If a diagonal $(a,b)$ appears $\ell$ times, then $T$ contains a path of length $\ell$ where the intermediate vertices have degree two in $T$, and are unlabeled. 
	
	Using this model, \cref{thm:global_bounding_simplex1} can be checked directly.  The weight $\wt_\PK(\tropP_\bullet)$ is equal to $r$, and the distance between any two vertices on the tree is bounded above by $r$.  It follows that the tree $T$ can be translated to lie inside the $r$-th dilate of the fundamental alcoved simplex.

	\section{Planar cross-ratios and planar basis}\label{sec: cross-ratios}
	In this section, we collect some of the basic properties of the planar cross-ratios $u_J$ of Definition~\ref{def:u-variable} and the planar basis $\h_J$ of Definition~\ref{def:planar-basis}. 
	
	\subsection{Planar basis}
	The lineality space is given by
	\begin{equation}\label{eq:lineality}
	\Lkn := \sp\Bigl(\sum_{I \ni i} e^I : i = 1, 2, \ldots, n\Bigr) \subset \R^{\binom{[n]}{k}},
	\end{equation}
	where $\{e^I : I \in \binom{[n]}{k}\}$ is the standard basis of $\R^{\binom{[n]}{k}}$. 	
	
		We may view the tropical Pl\"ucker function $p_I^t$ as the linear function on $\R^{\binom{[n]}{k}}$ that takes the coefficient of $e^I$.  The tropical function $u_J^t$ is then also a linear function 
	\eqref{eq:tropu}, obtained by the standard tropicalization substitution $(+, \times, \div) \mapsto (\min, +, -)$.  The following fundamental results were established in \cite{E20,E22,E25}.   See also \cite{KLZ} for recent work. 
	
	\begin{theorem}\label{thm:planarbasis}\
		\begin{enumerate}
			\item
			The elements $\{\h_J \mid J \in \binom{[n]}{k}\}$ form a basis of $\R^{\binom{[n]}{k}}$.  
			\item
			The elements $\{\h_J \mid J \text{ cyclic }\}$ form a basis of $L_{k,n}$.
			\item
			The elements $\{\h_J \mid J \in \bncyc\}$ form a basis of $\R^{\binom{[n]}{k}}/L_{k,n}$.
			\item 
			For $J, K \in \bncyc$, we have the duality
			$
			u_J^t(\h_K) = \delta_{J,K}
			$.
		\end{enumerate}
	\end{theorem}
	
	We also recall the following result from \cite{E22}, which states that the planar basis $\{\h_J \mid J \in \bncyc\}$ is among the rays of the positive tropical Grassmannian.
	\begin{thm}\label{thm:weak-separation-blade-arrangement}
		For each $J \in \bncyc$, the image of the one-dimensional cone $\R_{\geq 0} \h_J$ is a ray of $\Trop_{>0} X(k,n)$.  Given 
		$I,J \in \binom{[n]}{k}$, then $\mathfrak{h}_I + \mathfrak{h}_J \in \Trop_{>0} \Gr(k,n)$ 
		if and only if the pair $(I,J)$ is 
		weakly separated.  When $(I,J)$ is weakly separated, no positive linear combination $\alpha \h_I + \beta \h_J$ is a ray of $\Trop_{>0} \Gr(k,n)$.
	\end{thm}
	\begin{proof}
	In \cite{E22} the positroid subdivisions $\Delta(\h_I)$ of $\Delta(k,n)$ induced by $\h_I$ are described explicitly.  It is shown that the subdivisions $\Delta(\h_I)$ and $\Delta(\h_J)$ can be refined to a common positroid subdivision exactly when $I$ and $J$ are weakly separated.  Any positive linear combination $\alpha \h_I + \beta \h_J$ induces this common positroid subdivision, and cannot be a ray (since it refines another subdivision).
	\end{proof}
	
	\subsection{Planar cross-ratio}
	\begin{proposition}\label{prop:u-torus-invariance}
		The $u$-variable $u_J$ is well-defined on the configuration space 
		$X(k,n) $ if and only if $J$ is not a cyclic interval. 
	\end{proposition}
	\begin{proof}
		The torus $T = (\C^*)^n$ acts on the Pl\"ucker coordinates $p_I$ by the formula
		$$
		(z_1,\ldots,z_n) \cdot p_I = (\prod_{i\in I} z_i) p_I
		$$
		for $(z_1,\ldots,z_n) \in T$.  A rational function $f$ on $\Gr(k,n)$ descends to $X(k,n)$ if and only if $f$ is torus-invariant. This is verified for $u_J$ with $J \in \bncyc$ directly.
	\end{proof}
	
	\begin{proposition}\label{prop:u-product}
		The $\binom{n}{k}$ $u$-variables satisfy the multiplicative relation:
		\[
		\prod_{J \in \binom{[n]}{k}} u_J = 1.
		\]
	\end{proposition}
	\begin{proof}
		Follows from expanding $\prod_{J \in \binom{[n]}{k}} u_J$ as a Laurent monomial in Pl\"ucker coordinates.
	\end{proof}
	
	Recall that a cross-ratio is a torus-invariant Laurent monomial in Pl\"ucker coordinates.  
		
	\begin{proposition}\label{prop:u-basis}
		The set of planar cross-ratios $\{u_J : J \in \bncyc\}$ forms a 
		Laurent-monomial basis of cross-ratios on $X(k,n)$. That is, every cross-ratio on $X(k,n)$ can be written uniquely as 
		$\prod_{J \in \bncyc} u_J^{a_J}$ for some $a_J \in \Z$.
	\end{proposition}
	\begin{proof}
		The (multiplicative) group of Laurent monomials in Pl\"ucker coordinates is a free abelian group of rank $\binom{n}{k}$ with basis the Pl\"ucker coordinates.  The subgroup of torus-invariant Laurent monomials is free abelian of rank $\binom{n}{k}-n$.  That there are no multiplicative relations between the $\{u_J : J \in \bncyc\}$ follows from Theorem~\ref{thm:planarbasis}(4).  The integrality in Theorem~\ref{thm:planarbasis}(4) further implies the $\{u_J : J \in \bncyc\}$ form a basis of the group of torus-invariant Laurent monomials.
	\end{proof}
	
	For $\tropP_\bullet \in \mathbb{R}^{\binom{[n]}{k}}$, the 
	\emph{tropical $u$-variable} or \emph{tropical cross-ratio} is
	\begin{equation}\label{eq:tropu}
		u^t_J(\tropP_\bullet) = \sum_{M \in C_J} (-1)^{|J \cap M| - k - 1} \tropP_M,
	\end{equation}
	where $C_J$ is as defined in Definition~\ref{def:cubical-array}.  We will mainly evaluate the tropical cross-ratios
	on positive tropical Pl\"ucker vectors $\tropP_\bullet \in \Trop_{>0} \Gr(k,n)$.  By Theorem~\ref{thm:planarbasis}(4), for any $\tropP_\bullet \in \mathbb{R}^{\binom{[n]}{k}}$, we have
	\[
	\tropP_\bullet \equiv \sum_{J \in \bncyc} u^t_J(\tropP_\bullet) \, \h_J 
	\pmod{\Lkn}.
	\]
	
\subsection{Planar Basis is a Corank Function}\label{sec:corank}

In this section, we show that the planar basis can be viewed as the corank function of a particular positroid.  We recall the construction of decorated ordered set partitions from a $k$-element subset \cite{E22}.

\begin{defn}\label{def:DOSP}
	Let $J \in \binom{[n]}{k}$.  Arrange $[n]$ on a circle in the 
	standard cyclic order.  The elements of $J$ form $\ell$ maximal 
	runs of cyclically consecutive integers $I_1, \ldots, I_\ell$, 
	and each $I_a$ is preceded on the circle by a gap 
	$C_a \subset [n] \setminus J$ of cyclically consecutive integers.  
	Label the indices so that $1 \in C_1 \cup I_1$.  
	
	The \emph{decorated ordered set partition} associated to $J$ is 
	$(\r_J, \mathbf{S}_J) = ((r_1, \ldots, r_\ell), 
	(S_1, \ldots, S_\ell))$, where
	\[
	S_a := C_a \cup I_a, \qquad r_a := |I_a|.
	\]
	The blocks $(S_1, \ldots, S_\ell)$ partition $[n]$ into $\ell$ 
	cyclic intervals, and $\sum_{a=1}^\ell r_a = k$.
	
	In the terminology of \cite{E22}, $(\r_J, \mathbf{S}_J)$ \emph{is of type} $\Delta(k,n)$: it satisfies $1\le r_j\le \vert S_j\vert-1$ for all $J$. 
\end{defn}

\begin{example}\label{ex:non-corner}
	For $J = \{2, 5, 8\} \in \binom{[9]}{3}^\text{ncyc}$, we have $I_1 = \{2\}$, $I_2 = \{5\}$, 
	$I_3 = \{8\}$. The gaps are $C_1 = \{9, 1\}$, $C_2 = \{3, 4\}$, 
	$C_3 = \{6, 7\}$, giving sets $S_1 = \{9, 1, 2\}$, $S_2 = \{3, 4, 5\}$, 
	$S_3 = \{6, 7, 8\}$ with composition $\r = (1, 1, 1)$.
\end{example}

\begin{example}\label{ex:corner}
	For $J = \{2, 5, 6\} \in \binom{[6]}{3}^\text{ncyc}$, we have $I_1 = \{2\}$, 
	$I_2 = \{5, 6\}$. The gaps are $C_1 = \{1\}$, $C_2 = \{3, 4\}$, 
	giving sets $S_1 = \{1, 2\}$, $S_2 = \{3, 4, 5, 6\}$ with composition 
	$\r = (1, 2)$.
\end{example}

\begin{defn}\label{def:positroid-matroid}
	Let $(\r, \mathbf{S}) = ((r_1, \ldots, r_\ell), (S_1, \ldots, S_\ell))$ 
	be a decorated ordered set partition with $\sum_{a=1}^\ell r_a = k$ 
	and $\bigsqcup_{a=1}^\ell S_a = [n]$. We define the cyclic prefixes $T_a := S_1 \cup \cdots \cup S_a$ and partial capacities $\rho_a := r_1 + \cdots + r_a$. The matroid
	$M_{(\r, \mathbf{S})}$ on ground set $[n]$ has bases
	\[
	\B_{(\r, \mathbf{S})} = \left\{ B \in \binom{[n]}{k} : 
	|B \cap T_a| \geq \rho_a 
	\text{ for } a = 1, \ldots, \ell-1 \right\}.
	\]
\end{defn}
The matroid $M_{(\r, \mathbf{S})}$ is a \emph{positroid} since it is determined by rank conditions on cyclic intervals. By K\H{o}nig's theorem for transversal matroids, the corank of any subset $X \in \binom{[n]}{k}$ is the maximum capacity deficit across the cyclic prefixes:
\begin{equation}\label{eq:corank-max}
	\corank_{(\r, \mathbf{S})}(X) = \max_{1 \le a \le \ell-1} \max\bigl(0, \, \rho_a - |X \cap T_a|\bigr).
\end{equation}

In the following, we view the corank function as an element of $\R^{\binom{[n]}{k}}$.

\begin{thm}[Planar Basis as Corank Function]\label{thm:h=corank}
	Let $J \in \bncyc$ be a non-cyclic $k$-subset, and let $(\r, \mathbf{S}) = (\r_J, \mathbf{S}_J)$ be the decorated ordered set partition associated to $J$. Then
	\[
	\h_J \equiv \corank_{(\r, \mathbf{S})} \pmod{\Lkn}.
	\]
\end{thm}

\begin{proof}
	By Theorem~\ref{thm:planarbasis}(4), an element in $\R^{\binom{[n]}{k}}$ is uniquely determined modulo the lineality space $\Lkn$ by its evaluations on the tropical planar cross-ratios $u^t_K$ for all $K \in \bncyc$. We will prove that $u^t_K(\corank_J) = \delta_{K,J}$, where $\corank_J := \corank_{(\r_J, \mathbf{S}_J)}$.
	
	Let $I'_K = \{m \in K : m+1 \notin K\}$ denote the cyclic endpoints of $K$. By Definition \ref{def:u-variable}, the evaluation is the alternating sum over shifted sets $K_M = (K \setminus M) \cup \{m+1 : m \in M\}$ for $M \subseteq I'_K$:
	\begin{equation}\label{eq: corank-sum}
		u^t_K(\corank_J) = \sum_{M \subseteq I'_K} (-1)^{|M|+1} \corank_J(K_M).
	\end{equation}
	This alternating sum over the Boolean hypercube of shifts $\{0,1\}^{I'_K}$ is exactly the mixed finite difference operator. 
	
	Let $I'_J = \{j_1, \ldots, j_\ell\}$ be the cyclic endpoints of $J$, which correspond exactly to the right-boundaries of the prefix sets $T_a$. A shift $m \to m+1$ alters the intersection size $|K_M \cap T_a|$ if and only if $m = j_a$.
	
	\medskip
	\noindent \emph{The Diagonal Case ($K = J$)}.
	We evaluate $u^t_J(\corank_J)$. Because $J$ tightly satisfies the prefix constraints, $|J \cap T_a| = \rho_a$ and the baseline deficit is $0$.
	
	We partition the shifts $I'_J$ into the wrap-around endpoint $j_\ell$ (which crosses from $S_\ell$ into $S_1$) and the remaining interior endpoints $A' = I'_J \setminus \{j_\ell\}$. 
	
	If $j_\ell \in M$, the element shifts into $S_1 \subseteq T_a$. This strictly increases the intersection $|J_M \cap T_a|$ for all $a < \ell$, pushing the deficit to $\le -1$. No capacity constraints are violated, so $\corank_J(J_M) = 0$. Thus, all subsets containing $j_\ell$ perfectly vanish from the sum.
	
	If $M \subseteq A'$, any shift $j_a \in M$ removes an element from $T_a$, increasing its deficit by $1$. Because the prefixes are nested cyclically, any non-empty combination of these shifts creates a global maximum deficit of exactly $1$ across all constraints. Thus $\corank_J(J_M) = 1$ for all $\emptyset \neq M \subseteq A'$. When $M = \emptyset$, $J \in \B_J$, so $\corank_J(J) = 0$.
	
	The alternating sum collapses via the binomial theorem:
	\[
	u^t_J(\corank) = -(-1)^{0+1}(0) + \sum_{\emptyset \neq M \subseteq A'} (-1)^{|M|+1}(1) = 1 - (1-1)^{|A'|} = 1.
	\]
	(Since $J$ is non-cyclic, $\ell \geq 2$, so $|A'| \geq 1$ and the expansion is valid).
	
	\medskip
	
	\noindent \emph{Off-diagonal Case 1 ($K \neq J, I'_K \not\subseteq I'_J$)}.
	Choose $m^* \in I'_K \setminus I'_J$. Because $m^*$ is not a cyclic endpoint of $J$, it is not the boundary of any prefix $T_a$. The elements $m^*$ and $m^*+1$ lie strictly inside the interior of the exact same block. Thus, the intersection size $|K_M \cap T_a|$ is completely independent of whether $m^* \in M$. The fixed-point-free involution $\iota(M) = M \triangle \{m^*\}$ pairs terms with equal coranks and opposite signs, canceling the sum exactly to $0$.
	
	\medskip
	\noindent \emph{Off-diagonal Case 2 ($K \neq J, I'_K \subseteq I'_J$)}.
	By cyclic symmetry, we can select the starting point $1$ so that the wrap-around block $S_1$ has its left boundary at an endpoint of $K$. This guarantees $j_\ell \in I'_K$.
	
	Let $v_a = \rho_a - |K \cap T_a|$ be the unshifted baseline deficits of $K$. We define independent Boolean variables $m_a \in \{0,1\}$ indicating whether $j_a \in M$. A shift at $j_a$ removes an element from $T_a$, increasing its deficit by $m_a$. The shift at $j_\ell$ adds an element to $S_1 \subseteq T_a$, decreasing the deficit by $m_\ell$. The corank function on the hypercube is:
	\[
	f(M) = \max_{1 \le a \le \ell-1} \max\bigl(0, \, v_a + m_a - m_\ell\bigr).
	\]
	We must evaluate the mixed difference of $f$ over the variables $C = \{a : j_a \in I'_K\}$. Because $K$ is non-cyclic, $|C| \ge 2$.
	
	Let $C' = C \setminus \{\ell\}$. The function $f$ depends on $m_\ell$ and the variables in $C'$. By standard properties of mixed finite differences, the alternating sum over $C$ is non-zero only if $f$ has a top-degree cross-term involving all variables in $C$. This requires $v_c$ to equal the maximum $M_{\max} = \max_{a \in C'} v_a$ for all $c \in C'$, otherwise $f$ is independent of the non-maximizing variables.
	
	Assume $v_c = M_{\max}$ for all $c \in C'$. Let $Y = \max_{c \in C'} m_c$. The function collapses to $f(Y, m_\ell) = \max(0, M_{\max} + Y - m_\ell)$. The mixed difference over $m_\ell$ requires evaluating the jump:
	\[
	D(m_\ell) = f(1, m_\ell) - f(0, m_\ell) = \max(0, M_{\max} + 1 - m_\ell) - \max(0, M_{\max} - m_\ell).
	\]
	We evaluate the dependency on $m_\ell$:
	\begin{itemize}
		\item If $M_{\max} \ge 1$: $D(0) = (M_{\max}+1) - M_{\max} = 1$, and $D(1) = M_{\max} - (M_{\max}-1) = 1$. The difference $D(1) - D(0) = 0$.
		\item If $M_{\max} \le -1$: $D(0) = 0 - 0 = 0$, and $D(1) = 0 - 0 = 0$. The difference $D(1) - D(0) = 0$.
		\item If $M_{\max} = 0$: $D(0) = 1 - 0 = 1$, and $D(1) = 0 - 0 = 0$. The difference $D(1) - D(0) = -1 \neq 0$.
	\end{itemize}
	Therefore, the mixed difference is non-zero \emph{if and only if} $M_{\max} = 0$. 
	
	However, $M_{\max} = 0$ implies $v_a \le 0$ for all $a$, meaning $K \in \B_J$. The only basis of $M_J$ whose cyclic endpoints satisfy $I'_K \subseteq I'_J$ is $J$ itself. (Since its runs cannot end outside $I'_J$, its elements must be tightly packed against the right boundaries, recursively forcing $K=J$). 
	
	Since we assumed $K \neq J$, $M_{\max}$ cannot be $0$. Thus, the mixed difference evaluates strictly to $0$.
\end{proof}

\subsection{Face restriction maps}\label{sec:face-restriction}

For $\ell \in [n]$, define $\partial_\ell: \R^{\binom{[n]}{k}} \to \R^{\binom{[n-1]}{k-1}}$ by
\[
\partial_\ell(e^I) = \begin{cases}
	e^{I \setminus \{\ell\}} & \text{if } \ell \in I, \\
	0 & \text{if } \ell \notin I.
\end{cases}
\]
where we identify $[n-1]$ with $[n] \setminus \{\ell\}$ in the natural way.  This map can be viewed as the restriction to the face $x_\ell = 1$ of $\Delta(k,n)$.  We write $\h^{(\ell)}_K$ for the planar basis element in $\R^{\binom{[n-1]}{k-1}}$ indexed by $K$.

\begin{lemma}[{\cite[Lemma 24]{E22}}] \label{lem:pl}
Let $J \in \binom{[n]}{k}^{\ncyc}$ and $\ell \in [n]$.  Let $m \geq 0$ be minimal such that $\ell + m \in J$ (indices mod $n$).  Then
\[
\partial_{\ell}(\h_J) \equiv \h^{(\ell)}_{J'} \pmod{\mathrm{lineality}},
\]
where $J' = J \setminus \{\ell+m\}$.
\end{lemma} 

More generally, for an order-preserving injection $\iota: [n'] \hookrightarrow [n]$ with image $T$, let $\partial_\iota : \R^{\binom{[n]}{k}} \to \R^{\binom{[n']}{k-(n-n')}}$ send $e^I$ to $e^{\iota^{-1}(I \cap T)}$ if $I \supset [n] \setminus T$ and to $0$ otherwise; equivalently, $\partial_\iota = \prod_{\ell \in [n] \setminus T} \partial_{\ell}$. 

\begin{cor}\label{cor:multi-del}
For $J \in \binom{[n]}{k}^{\ncyc}$,
\[
\partial_{\iota}(\h_J) \equiv \h^{(\iota)}_{\iota^{-1}(J^*)} \pmod{\mathrm{lineality}},
\]
where $J^* \subseteq T$ is obtained by iterating Lemma~\ref{lem:pl} over $[n] \setminus T$.
\end{cor}

For completeness, we include the other family of face restriction maps.  For $\ell \in [n]$, define $\partial_{\ell_0}: \R^{\binom{[n]}{k}} \to \R^{\binom{[n-1]}{k}}$ by
\[
\partial_{\ell_0}(e^I) = \begin{cases}
	e^{I} & \text{if } \ell \notin I, \\
	0 & \text{if } \ell \in I.
\end{cases}
\] 
where we identify $[n-1]$ with $[n] \setminus \{\ell\}$ in the natural way.  This map can be viewed as the restriction to the face $x_\ell = 0$ of $\Delta(k,n)$.  We write $\h^{(\ell_0)}_K$ for the planar basis element in $\R^{\binom{[n-1]}{k}}$ indexed by $K$.

\begin{lem}\label{lem:face-restriction}
Let $J \in \binom{[n]}{k}^{\ncyc}$ and $\ell \in [n]$.  Let $m \geq 0$ be minimal such that $\ell - m \notin J$ (indices mod $n$).  Then
\[
\partial_{\ell_0}(\h_J) \equiv \h^{(\ell_0)}_{J'} \pmod{\mathrm{lineality}},
\]
where $J' = J$ if $m = 0$, and $J' = (J \setminus \{\ell\}) \cup \{\ell - m\}$ otherwise.
\end{lem}

\begin{proof}
By cyclic symmetry, assume $\ell = n$.  By Theorem~\ref{thm:h=corank} it suffices to show $\partial_{n_0}(\corank_{(\r, \mathbf{S})}) = \corank_{(\r', \mathbf{S}')}$, where $(\r, \mathbf{S})$ and $(\r', \mathbf{S}')$ are the decorated partitions of $J$ and $J'$.  For any $I \subseteq [n-1]$ with $|I| = k$, matroid deletion gives $\corank_M(I) = \corank_{M \setminus n}(I)$ since $n \notin I$.  The identity $M \setminus n = M_{(\r', \mathbf{S}')}$ is a direct check: in each case of $J'$, the block $S_\ell$ loses one element, with $r'_a = r_a$ and the earlier constraints unchanged.
\end{proof}

More generally, for an order-preserving injection $\iota: [n'] \hookrightarrow [n]$ with image $T$, let $\partial_{\iota_0} : \R^{\binom{[n]}{k}} \to \R^{\binom{[n']}{k}}$ send $e^I$ to $e^{\iota^{-1}(I)}$ if $I \subseteq T$ and to $0$ otherwise; equivalently, $\partial_{\iota_0} = \prod_{\ell \in [n] \setminus T} \partial_{\ell_0}$.

\begin{cor}\label{cor:multi-deletion}
For $J \in \binom{[n]}{k}^{\ncyc}$,
\[
\partial_{\iota_0}(\h_J) \equiv \h^{(\iota_0)}_{\iota^{-1}(J^*)} \pmod{\mathrm{lineality}},
\]
where $J^* \subseteq T$ is obtained by iterating Lemma~\ref{lem:face-restriction} over $[n] \setminus T$.
\end{cor}

\begin{example}[$(k, n, n') = (3, 10, 6)$]
Let $\iota: [6] \hookrightarrow [10]$ have image $T = \{1, 3, 5, 6, 7, 8\}$, and take $J = \{3, 6, 9\}$ with decorated ordered set partition
\[
\big((10,1,2,3)_1, (4,5,6)_1, (7,8,9)_1\big).
\]
The deleted indices $[10] \setminus T = \{2, 4, 9, 10\}$ act on this partition as follows: $2, 4, 10 \in [n] \setminus J$ simply shrink the gaps of their respective blocks, while $9 \in J$ is replaced, by Lemma~\ref{lem:face-restriction}, by its cyclic predecessor $8 \in T \setminus J$.  Iterating gives the partition in $T$
\[
\big((1,3)_1, (5,6)_1, (7,8)_1\big),\qquad \text{equivalently } \big((1,2)_1, (3,4)_1, (5,6)_1\big) \text{ in } [6] \text{ via } \iota^{-1},
\]
so $J^* = \{3, 6, 8\}$, $\iota^{-1}(J^*) = \{2, 4, 6\}$, and
\[
\partial_{\iota_0}(\h_{3,6,9}) \equiv \h^{(\iota)}_{2,4,6} \pmod{\mathrm{lineality}}.
\]
\end{example}

\begin{example}
With the same $\iota$ and $T$, take $J = \{4, 9, 10\}$ with decorated ordered set partition $$((1,2,3,4)_1, (5,6,7,8,9,10)_2).$$  Iterating Lemma~\ref{lem:face-restriction} yields the partition in $T$
\[
\big((1,3)_1, (5,6,7,8)_2\big),\qquad \text{equivalently } \big((1,2)_1, (3,4,5,6)_2\big) \text{ in } [6],
\]
so $\partial_{\iota_0}(\h_{4,9,10}) \equiv \h^{(\iota)}_{2,5,6}$.  Here the deletion of $10 \in J$ cycles back past $8 \in J$ to reach $7 \in T \setminus J$; the decoration $r_2 = 2$ is preserved, with the active element $I_2$ shifting from $\{9, 10\}$ to $\{7, 8\}$ at the expense of the gap $C_2$.
\end{example}

	\section{Noncrossing Combinatorics}\label{sec:noncrossing}
\def\te{{\tilde e}}
\subsection{Preliminaries}
We write $\te_{i,1},\ldots,\te_{i,n-k}$ for the basis vectors in the $i$-th copy of $\R^{n-k}$ in $\Rkn$.
We write $e_{i,1},\ldots,e_{i,n-k}$ for the images of the basis vectors in the $i$-th copy of $\T^{n-k-1}$ in $\Tkn$.
For 
$J = \{j_1 < \cdots < j_k\} \in \binom{[n]}{k}$, define 
$\tvn_J \in \Rkn_\Z$ and $\vn_J \in \Tkn_\Z$ \cite{CE24,E21} by
\[
\tvn_J = \sum_{i=1}^{k} \sum_{s = 1}^{j_i - i} \te_{i,s}, \qquad 
\vn_J = \sum_{i=1}^{k-1} \sum_{s = j_i - (i-1)}^{j_{i+1} - (i+1)} e_{i,s}.
\]
Here, we adopt the convention that $\te_{i,j} = 0$ and $e_{i,j}=0$ unless $j\ge 1$.  Define a linear projection $\Phi: \Rkn \to \Tkn$ by 
\[
\Phi(\te_{i,j})  = \begin{cases}
	-e_{i,j}, & i = 1, \\
	-e_{i,j} + e_{i-1,j}, & 2 \le i \le k-1, \\
	e_{k-1,j}, & i = k.
\end{cases}
\]
\begin{example}
	In the case $(k,n) = (3,5)$ we have
	$$\tvn_{123}=0,\ \tvn_{124} = \te_{3,1},\ \tvn_{125} = \te_{3,1} + \te_{3,2},\ \tvn_{134} = \te_{2,1} + \te_{3,1},\ \tvn_{135} = \te_{2,1} + \te_{3,[1,2]},\ \tvn_{145} = \te_{2,[1,2]} + \te_{3,[1,2]},$$
	$$\tvn_{234} = \te_{1,1} + \te_{2,1} + \te_{3,1},\ \tvn_{235} = \te_{1,1} + \te_{2,1} + \te_{3,[1,2]},\ \tvn_{245} = \te_{1,1} + \te_{2,[1,2]} + \te_{3,[1,2]},\ \tvn_{345} = \te_{1,[1,2]} + \te_{2,[1,2]} + \te_{3,[1,2]},$$
	and 
	$$\t_{123} =0,\ \t_{124} = e_{2,1},\ \t_{125} = e_{2,[1,2]} \equiv 0,\ \t_{134} = e_{1,1},\ \t_{135} = e_{1,1} + e_{2,2},\ \t_{145} = e_{1,[1,2]} \equiv 0,$$
	$$\t_{234} = 0,\ \t_{235} = e_{2,2},\ \t_{245} = e_{1,2},\ \t_{345} = 0.$$
\end{example}

\begin{proposition}\label{prop:tJ}\
	\begin{enumerate}
		\item
		The set $\{\tvn_J \mid J \in \binom{[n]}{k}\}$ spans $\Rkn$.
		\item
		The set $\{\vn_J \mid J \in \binom{[n]}{k}^\ncyc\}$ spans $\Tkn$.
		\item
		We have $\Phi(\tvn_J) = \vn_J.$
		\item
		The kernel of $\Phi$ is spanned by the set of $\tvn_{J}$ where $J$ is a cyclic interval.
	\end{enumerate}
\end{proposition}

\begin{proof}
	(1) and (3) are checked directly.  (2) follows from (1), (3), the surjectivity of $\Phi$, and that $\vn_J = 0$ for $J$ a cyclic interval.  It follows that $\tvn_J \in \ker(\Phi)$ for $J$ a cyclic interval.  We have $\tvn_{[1,k]} = 0$.  The remaining $n-1$ vectors in the set $\{\tvn_J \mid J \text{ a cyclic interval}\}$ are linearly independent.  Since $\dim(\ker(\Phi)) = n-1$, we obtain (4).
\end{proof}

	\subsection{Noncrossing fan}
	\def\tNC{{\widetilde{\NC}}}

	In this subsection we prove Theorem~\ref{thm:NCfan}.
	
	Denote by $\tNC_{k,n}$ the poset of all collections of pairwise 
	noncrossing $k$-element subsets, ordered by inclusion; 
	again, it is known that the maximal (by inclusion) collections each have 
	exactly $k(n-k)+1$ $k$-element subsets; see \cite{PKPS10}, \cite{SSW17}. 
	One of the main results of \cite{SSW17} is the following.
	\begin{theorem}
		The maximal faces of $\tNC_{k,n}$ give a regular unimodular triangulation $\mathcal{T}_{k,n}$ of the order polytope $O_{k,n} := \conv(\tvn_J \mid J \in \binom{[n]}{k})$.
	\end{theorem}
	
	\begin{proof}[Proof of Theorem~\ref{thm:NCfan}]
		
		By Proposition~\ref{prop:tJ}, the kernel of $\Phi: \Rkn \to \Tkn$ is spanned by the set $\{\tvn_J : J \text{ is a cyclic interval}\}$.  The $n$ cyclic intervals form a noncrossing pair with every $I \in \binom{[n]}{k}$.  Therefore, every maximal simplex $\sigma$ of $\mathcal{T}_{k,n}$ contains all $n$ cyclic vertices $\{\tvn_J : J \text{ cyclic}\}$.
		
		We consider the image polytope $\Phi(O_{k,n})$.  The order polytope $O_{k,n}$ has a unique facet $F = \{x_{k,1} = 1\}$ not passing through the origin; see \cite[Section 4]{SSW17}, noting that our conventions differ.  The normal vector to this facet is contained in the kernel of $\Phi$.  It follows that the image $\Phi(O_{k,n})$ contains the origin in its interior.  The projection $\Phi$ maps each maximal simplex $\sigma$ of $\mathcal{T}_{k,n}$ to a simplex $\Phi(\sigma)$ of full dimension $(k-1)(n-k-1)$. The collection $\{\Phi(\sigma) : \sigma \in \mathcal{T}_w \text{ maximal}\}$ forms a unimodular triangulation of the image $\Phi(O_{k,n})$, which has vertices 
		$\{\vn_J = \Phi(\tvn_J) : J \in \bncyc\}$.  Each simplex $\Phi(\sigma)$ of this triangulation has one vertex at the origin, and by taking the cones spanned by these simplices, we obtain a complete, unimodular fan $F_{\NC_{k,n}}$ in $\Tkn$ with cones indexed by faces of the noncrossing complex $\NC_{k,n}$.  This proves Theorem~\ref{thm:NCfan}.	\end{proof}

	We call $F_{\NC_{k,n}}$ the \emph{noncrossing fan}.  We record the following for future use.

	\begin{corollary}\label{cor:nc-fan-rays}
	The rays of the noncrossing fan $F_{\NC_{k,n}}$ are exactly
		\[
		\left\{\, \mathbb{R}_{\geq 0} \cdot \vn_J : J \in \bncyc \,\right\}.
		\]
	\end{corollary}

	\subsection{Proof of Theorem~\ref{thm:noncrossingmain}}
	Since $\h_J$ for $J$ cyclic span $L_{k,n}$ (Theorem~\ref{thm:planarbasis}) and $\Phi(\h_J) =\vn_J = 0$ for $J$ cyclic, the map $\Psi: \R^{\binom{[n]}{k}} \to \Tkn$ descends to a map $\Psi: \R^{\binom{[n]}{k}}/L_{k,n} \to \Tkn$.  The vectors $\h_J$ span $\R^{\binom{[n]}{k}}$ and the vectors $\vn_J$ span $\Tkn$.  Comparing Theorem~\ref{thm:planarbasis}(4) with Theorem~\ref{thm:duality}, we deduce that $\rho \circ \Psi$ is the identity.  Since the positive parametrization \eqref{eq:Mx} is injective, so is $\rho$, and we conclude that the maps $\Psi$ and $\rho$ are inverse bijections between $\Trop_{>0} X(k,n)$ and $\Tkn$.
	
	The integrality in Theorem~\ref{thm:planarbasis}(4) and Theorem~\ref{thm:duality} implies that $\Psi$ and $\rho$ restrict to inverse bijections between $(\Trop_{>0} X(k,n))(\Z)$ and $\Tkn_\Z$.  Composing with \cref{thm:NCfan} we obtain the stated bijection between $(\Trop_{>0} X(k,n))(\Z)$ and noncrossing tableaux.

Define $\psi: \R^{\binom{[n]}{k}} \to \Rkn$ by $\psi(\h_J) = \tvn_J$.  Thus $\Psi = \Phi \circ \psi$.

\begin{corollary}\label{cor: trivial intersection general}
	We have $\Trop_{>0}\Gr(k,n) \cap \ker(\psi) = \Lkn$.
\end{corollary}

\begin{proof}
	Suppose $\tropP_\bullet \in \Trop_{>0}\Gr(k,n) \cap \ker(\psi) $.  We have $\psi(\tropP_\bullet) = 0$, hence $\Psi(\tropP_\bullet) = \Phi(\psi(\tropP_\bullet)) = 0$.
	
	By Theorem~\ref{thm:noncrossingmain}, $\Psi$ restricts to a bijection on $\Trop_{>0}X(k,n)$.  The only element mapping to $0$ is the equivalence class of the lineality space.  Thus $\tropP_\bullet \in \Lkn$.
\end{proof}

\section{The duality theorem}\label{sec:dualityproof}

\subsection{The Positive Parameterization}

The positive configuration space $X(k,n)_{>0} = \Gr(k,n)_{>0}/(\R^n_{>0})$ has dimension $(k-1)(n-k-1)$. We parameterize it using variables $\{x_{\ell,t} : \ell \in [k-1], \, t \in [n-k]\}$, which we arrange into a $(k-1) \times (n-k)$ array.

\begin{definition}[Positive Parameterization Matrix]\label{def:positive-param}
	Define a $(k-1) \times (n-k)$ matrix $M_{k,n}$ with polynomial entries
	\[
	m_{i,j} = (-1)^{k-i} \sum_{1 \leq b_i \leq b_{i+1} \leq \cdots \leq b_{k-1} \leq j} x_{i,b_i} \, x_{i+1,b_{i+1}} \cdots x_{k-1,b_{k-1}}.
	\]
	The full $k \times n$ parameterization matrix $M$ is constructed by embedding $M_{k,n}$ as the upper-right $(k-1) \times (n-k)$ block:
	\[
	M = \begin{pmatrix}
		1 & 0 & \cdots & 0 & 0&m_{1,1} & m_{1,2} & \cdots & m_{1,n-k} \\
		0 & 1 & \cdots & 0 & 0&m_{2,1} & m_{2,2} & \cdots & m_{2,n-k} \\
		\vdots & \vdots & \ddots& \vdots & \vdots & \vdots & \vdots & & \vdots \\
		0 & 0 & \cdots & 1 & 0&m_{k-1,1} & m_{k-1,2} & \cdots & m_{k-1,n-k} \\
		0 & 0 & \cdots & 0 & 1 & 1 & \cdots & 1
	\end{pmatrix}.
	\]
	The first $k$ columns form the identity matrix $I_k$, and the bottom row of the right block consists entirely of $1$'s.
\end{definition}

\begin{example}[The case $(k,n) = (3,6)$]\label{ex:param-36}
	For $(k,n) = (3,6)$, the parameterization involves a $2 \times 3$ array of variables $\{x_{1,1}, x_{1,2}, x_{1,3}, x_{2,1}, x_{2,2}, x_{2,3}\}$. The matrix entries are:
	\begin{align*}
		m_{1,1} &= x_{1,1} x_{2,1}, \\
		m_{1,2} &= x_{1,1} x_{2,1} + x_{1,1} x_{2,2} + x_{1,2} x_{2,2}, \\
		m_{1,3} &= x_{1,1}(x_{2,1} + x_{2,2} + x_{2,3}) + x_{1,2}(x_{2,2} + x_{2,3}) + x_{1,3} x_{2,3}, \\
		m_{2,j} &= -(x_{2,1} + x_{2,2} + \cdots + x_{2,j}).
	\end{align*}
	The full $3 \times 6$ matrix is:
	\[
	M = \begin{pmatrix}
		1 & 0 & 0 & m_{1,1} & m_{1,2} & m_{1,3} \\
		0 & 1 & 0 & m_{2,1} & m_{2,2} & m_{2,3} \\
		0 & 0 & 1 & 1 & 1 & 1
	\end{pmatrix}.
	\]
\end{example}

The Pl\"ucker coordinates $p_I = p_I(\x)$ for $I \in \binom{[n]}{k}$ are the $k \times k$ minors of $M(\x)$.  By the Lindstr\"om--Gessel--Viennot lemma, or by Postnikov's theory \cite{Pos06}, each $p_I$ is a weighted sum over non-intersecting path families in a ladder network.

\begin{prop}
	The matrix $M(\x)$ induces a homeomorphism
	$$
	(\R_{>0}^{n-k}/\R_{>0})^{k-1} \xrightarrow{\,\cong\,} X(k,n)_{>0}.
	$$
\end{prop}

The ladder network for $(k,n)$ is a planar directed graph with $(k-1)$ horizontal levels, indexed by $\ell = 1, 2, \ldots, k-1$, and $(n-k)$ horizontal positions, indexed by $t = 1, 2, \ldots, n-k$.  There are $k$ source nodes on the left, one for each row $r \in [k]$, and $(n-k)$ sink nodes on the right, at exit positions $t = 1, \ldots, n-k$ (corresponding to columns $k+1, \ldots, n$ of the matrix $M$).  At each lattice point $(\ell, t)$ there is a vertical edge from level $\ell$ to level $\ell+1$, weighted by $x_{\ell,t}$; all horizontal edges have weight~$1$.  Source $r$ enters the network at level~$r$.  See Figures~\ref{fig:ladder-network} and \ref{fig:ladder-37}.

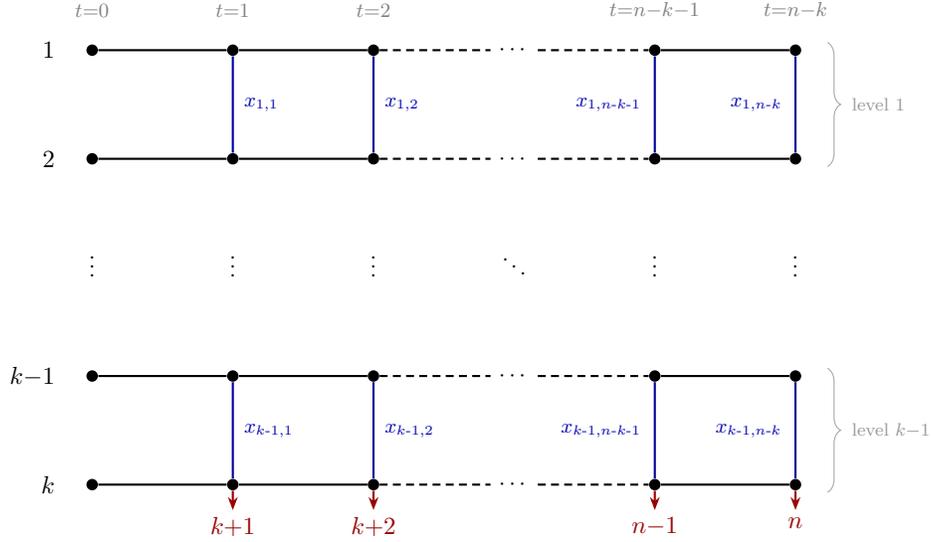
\begin{figure}[ht]
	\centering
	\begin{tikzpicture}[
		every node/.style={font=\small},
		nd/.style={circle, fill=black, inner sep=1.8pt},
		hedge/.style={thick},
		vedge/.style={thick, blue!60!black},
		wlabel/.style={font=\footnotesize, text=blue!70!black},
		scale=0.85, transform shape
		]
		
		\def\hs{2.2}
		\def\vs{1.7}
		
		\node[nd] (A0) at (0,0) {};
		\node[nd] (A1) at (\hs,0) {};
		\node[nd] (A2) at (2*\hs,0) {};
		\node     (Ad) at (3*\hs,0) {$\cdots$};
		\node[nd] (A4) at (4*\hs,0) {};
		\node[nd] (A5) at (5*\hs,0) {};
		\draw[hedge] (A0) -- (A1);
		\draw[hedge] (A1) -- (A2);
		\draw[hedge, densely dashed] (A2) -- (Ad);
		\draw[hedge, densely dashed] (Ad) -- (A4);
		\draw[hedge] (A4) -- (A5);
		
		\node[nd] (B0) at (0,-\vs) {};
		\node[nd] (B1) at (\hs,-\vs) {};
		\node[nd] (B2) at (2*\hs,-\vs) {};
		\node     (Bd) at (3*\hs,-\vs) {$\cdots$};
		\node[nd] (B4) at (4*\hs,-\vs) {};
		\node[nd] (B5) at (5*\hs,-\vs) {};
		\draw[hedge] (B0) -- (B1);
		\draw[hedge] (B1) -- (B2);
		\draw[hedge, densely dashed] (B2) -- (Bd);
		\draw[hedge, densely dashed] (Bd) -- (B4);
		\draw[hedge] (B4) -- (B5);
		
		\foreach \x in {0,1,2,4,5} {
			\pgfmathsetmacro{\xc}{\x*\hs}
			\node at (\xc, -2*\vs+0.12) {$\vdots$};
		}
		\node at (3*\hs, -2*\vs+0.12) {$\ddots$};
		
		\node[nd] (C0) at (0,-3*\vs) {};
		\node[nd] (C1) at (\hs,-3*\vs) {};
		\node[nd] (C2) at (2*\hs,-3*\vs) {};
		\node     (Cd) at (3*\hs,-3*\vs) {$\cdots$};
		\node[nd] (C4) at (4*\hs,-3*\vs) {};
		\node[nd] (C5) at (5*\hs,-3*\vs) {};
		\draw[hedge] (C0) -- (C1);
		\draw[hedge] (C1) -- (C2);
		\draw[hedge, densely dashed] (C2) -- (Cd);
		\draw[hedge, densely dashed] (Cd) -- (C4);
		\draw[hedge] (C4) -- (C5);
		
		\node[nd] (D0) at (0,-4*\vs) {};
		\node[nd] (D1) at (\hs,-4*\vs) {};
		\node[nd] (D2) at (2*\hs,-4*\vs) {};
		\node     (Dd) at (3*\hs,-4*\vs) {$\cdots$};
		\node[nd] (D4) at (4*\hs,-4*\vs) {};
		\node[nd] (D5) at (5*\hs,-4*\vs) {};
		\draw[hedge] (D0) -- (D1);
		\draw[hedge] (D1) -- (D2);
		\draw[hedge, densely dashed] (D2) -- (Dd);
		\draw[hedge, densely dashed] (Dd) -- (D4);
		\draw[hedge] (D4) -- (D5);
		
		\draw[vedge] (A1) -- (B1) node[midway, right=1pt, wlabel] {$x_{1,1}$};
		\draw[vedge] (A2) -- (B2) node[midway, right=1pt, wlabel] {$x_{1,2}$};
		\draw[vedge] (A4) -- (B4) node[midway, left=2pt, wlabel] {$x_{1,n\text{-}k\text{-}1}$};
		\draw[vedge] (A5) -- (B5) node[midway, left=2pt, wlabel] {$x_{1,n\text{-}k}$};
		
		\draw[vedge] (C1) -- (D1) node[midway, right=1pt, wlabel] {$x_{k\text{-}1,1}$};
		\draw[vedge] (C2) -- (D2) node[midway, right=1pt, wlabel] {$x_{k\text{-}1,2}$};
		\draw[vedge] (C4) -- (D4) node[midway, left=2pt, wlabel] {$x_{k\text{-}1,n\text{-}k\text{-}1}$};
		\draw[vedge] (C5) -- (D5) node[midway, left=2pt, wlabel] {$x_{k\text{-}1,n\text{-}k}$};
		
		\node[left=10pt of A0, font=\normalsize] {$1$};
		\node[left=10pt of B0, font=\normalsize] {$2$};
		\node[left=10pt of C0, font=\normalsize] {$k{-}1$};
		\node[left=10pt of D0, font=\normalsize] {$k$};
		
		\draw[-{Stealth[length=5pt]}, red!60!black, thick] (D1) -- ++(0,-0.4);
		\draw[-{Stealth[length=5pt]}, red!60!black, thick] (D2) -- ++(0,-0.4);
		\draw[-{Stealth[length=5pt]}, red!60!black, thick] (D4) -- ++(0,-0.4);
		\draw[-{Stealth[length=5pt]}, red!60!black, thick] (D5) -- ++(0,-0.4);
		\node[below=8pt of D1, font=\normalsize, red!60!black] {$k{+}1$};
		\node[below=8pt of D2, font=\normalsize, red!60!black] {$k{+}2$};
		\node[below=8pt of D4, font=\normalsize, red!60!black] {$n{-}1$};
		\node[below=8pt of D5, font=\normalsize, red!60!black] {$n$};
		
		\node[above=8pt of A0, font=\footnotesize, gray] {$t{=}0$};
		\node[above=8pt of A1, font=\footnotesize, gray] {$t{=}1$};
		\node[above=8pt of A2, font=\footnotesize, gray] {$t{=}2$};
		\node[above=8pt of A4, font=\footnotesize, gray] {$t{=}n{-}k{-}1$};
		\node[above=8pt of A5, font=\footnotesize, gray] {$t{=}n{-}k$};
		
		\draw[decorate, decoration={brace, amplitude=5pt}, gray!70]
		($(A5)+(0.5,0.12)$) -- ($(B5)+(0.5,-0.12)$)
		node[midway, right=7pt, font=\scriptsize, gray!80] {level $1$};
		\draw[decorate, decoration={brace, amplitude=5pt}, gray!70]
		($(C5)+(0.5,0.12)$) -- ($(D5)+(0.5,-0.12)$)
		node[midway, right=7pt, font=\scriptsize, gray!80] {level $k{-}1$};
		
	\end{tikzpicture}
	\caption{The ladder network for $\Gr(k,n)$.
		There are $k$ horizontal rails and $k-1$ levels of vertical edges.
		Horizontal edges have weight $1$; the vertical edge at level $\ell$, position $t$ has weight $x_{\ell,t}$.
		Source $r$ (on the left) traverses levels $r, r{+}1, \ldots, k{-}1$, contributing $(k-r)$ factors.
		Sink $j$ (for $j = k{+}1, \ldots, n$) exits at position $t = j - k$ on the bottom rail.}
	\label{fig:ladder-network}
\end{figure}

\begin{figure}[ht]
	\centering
	\begin{tikzpicture}[
		every node/.style={font=\small},
		nd/.style={circle, fill=black, inner sep=2pt},
		hedge/.style={thick, -{Stealth[length=4pt]}},
		vedge/.style={thick, blue!60!black, -{Stealth[length=4pt]}},
		wlabel/.style={font=\footnotesize, text=blue!70!black},
		scale=0.85, transform shape
		]
		
		\def\hs{2.2}
		\def\vs{2.0}
		
		\node[nd] (A0) at (0,0) {};
		\node[nd] (A1) at (\hs,0) {};
		\node[nd] (A2) at (2*\hs,0) {};
		\node[nd] (A3) at (3*\hs,0) {};
		\node[nd] (A4) at (4*\hs,0) {};
		\draw[hedge] (A0) -- (A1);
		\draw[hedge] (A1) -- (A2);
		\draw[hedge] (A2) -- (A3);
		\draw[hedge] (A3) -- (A4);
		
		\node[nd] (B0) at (0,-\vs) {};
		\node[nd] (B1) at (\hs,-\vs) {};
		\node[nd] (B2) at (2*\hs,-\vs) {};
		\node[nd] (B3) at (3*\hs,-\vs) {};
		\node[nd] (B4) at (4*\hs,-\vs) {};
		\draw[hedge] (B0) -- (B1);
		\draw[hedge] (B1) -- (B2);
		\draw[hedge] (B2) -- (B3);
		\draw[hedge] (B3) -- (B4);
		
		\node[nd] (C0) at (0,-2*\vs) {};
		\node[nd] (C1) at (\hs,-2*\vs) {};
		\node[nd] (C2) at (2*\hs,-2*\vs) {};
		\node[nd] (C3) at (3*\hs,-2*\vs) {};
		\node[nd] (C4) at (4*\hs,-2*\vs) {};
		\draw[hedge] (C0) -- (C1);
		\draw[hedge] (C1) -- (C2);
		\draw[hedge] (C2) -- (C3);
		\draw[hedge] (C3) -- (C4);
		
		\draw[vedge] (A1) -- (B1) node[midway, right=2pt, wlabel] {$x_{1,1}$};
		\draw[vedge] (A2) -- (B2) node[midway, right=2pt, wlabel] {$x_{1,2}$};
		\draw[vedge] (A3) -- (B3) node[midway, right=2pt, wlabel] {$x_{1,3}$};
		\draw[vedge] (A4) -- (B4) node[midway, left=2pt, wlabel] {$x_{1,4}$};
		
		\draw[vedge] (B1) -- (C1) node[midway, right=2pt, wlabel] {$x_{2,1}$};
		\draw[vedge] (B2) -- (C2) node[midway, right=2pt, wlabel] {$x_{2,2}$};
		\draw[vedge] (B3) -- (C3) node[midway, right=2pt, wlabel] {$x_{2,3}$};
		\draw[vedge] (B4) -- (C4) node[midway, left=2pt, wlabel] {$x_{2,4}$};
		
		\node[left=12pt of A0, font=\normalsize] {$1$};
		\node[left=12pt of B0, font=\normalsize] {$2$};
		\node[left=12pt of C0, font=\normalsize] {$3$};
		
		\draw[-{Stealth[length=5pt]}, red!60!black, thick] (C1) -- ++(0,-0.5);
		\draw[-{Stealth[length=5pt]}, red!60!black, thick] (C2) -- ++(0,-0.5);
		\draw[-{Stealth[length=5pt]}, red!60!black, thick] (C3) -- ++(0,-0.5);
		\draw[-{Stealth[length=5pt]}, red!60!black, thick] (C4) -- ++(0,-0.5);
		\node[below=10pt of C1, font=\normalsize, red!60!black] {$4$};
		\node[below=10pt of C2, font=\normalsize, red!60!black] {$5$};
		\node[below=10pt of C3, font=\normalsize, red!60!black] {$6$};
		\node[below=10pt of C4, font=\normalsize, red!60!black] {$7$};
		
		\node[above=10pt of A0, font=\footnotesize, gray] {$t{=}0$};
		\node[above=10pt of A1, font=\footnotesize, gray] {$t{=}1$};
		\node[above=10pt of A2, font=\footnotesize, gray] {$t{=}2$};
		\node[above=10pt of A3, font=\footnotesize, gray] {$t{=}3$};
		\node[above=10pt of A4, font=\footnotesize, gray] {$t{=}4$};
		
		\draw[decorate, decoration={brace, amplitude=6pt}, gray!70]
		($(A4)+(0.5,0.15)$) -- ($(B4)+(0.5,-0.15)$)
		node[midway, right=8pt, font=\small, gray!80] {level $1$};
		\draw[decorate, decoration={brace, amplitude=6pt}, gray!70]
		($(B4)+(0.5,0.15)$) -- ($(C4)+(0.5,-0.15)$)
		node[midway, right=8pt, font=\small, gray!80] {level $2$};
		
	\end{tikzpicture}
	\caption{The ladder network for $(k,n) = (3,7)$, with $k-1 = 2$ levels and $n-k = 4$ positions.
		Source $1$ crosses both levels (contributing factors $x_{1,t_1} \cdot x_{2,t_2}$),
		source $2$ crosses level $2$ only (contributing $x_{2,t_2}$), and
		source $3$ traverses the bottom rail with weight $1$.}
	\label{fig:ladder-37}
\end{figure}
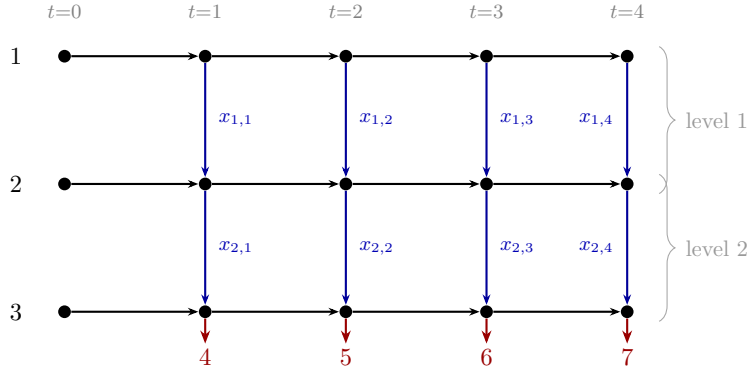

We now describe how the Pl\"ucker coordinates are computed from this network.  Let $I \in \binom{[n]}{k}$, and write $I_{\leq k} = I \cap [k]$ and $I_{>k} = I \cap [k+1, n]$.  Since the first $k$ columns of $M$ form the identity matrix $I_k$, cofactor expansion along these columns shows that the Pl\"ucker coordinate $p_I$ reduces to a minor involving only the active sources $[k] \setminus I_{\leq k}$ and the active columns $I_{>k}$.  By the Lindstr\"om--Gessel--Viennot lemma, this reduced minor equals a sum over non-intersecting path families in the ladder network.

\begin{proposition}\label{prop: pJ formula general}
	For $I \in \binom{[n]}{k}$, the Pl\"ucker coordinate $p_I(\x)$ is given by
	\begin{equation}\label{eq: pJ formula general}
		p_I(\x) = \sum_{P \in \mathcal{P}_I} \wt(P),
	\end{equation}
	where $\mathcal{P}_I$ denotes the set of non-intersecting path families in the ladder network connecting the active sources $[k] \setminus I_{\leq k}$ to the active sinks at positions $\{j - k : j \in I_{>k}\}$, and $\wt(P)$ denotes the product of all vertical edge weights traversed by the paths in~$P$.
\end{proposition}

Tropicalizing the path formula yields the following.

\begin{lemma}\label{lem: tropical pJ general}
	For any $I \in \binom{[n]}{k}$,
	$$\tropP_I(\y) = \min_{P \in \mathcal{P}_I} \wt^{\mathrm{trop}}(P),$$
	where $\wt^{\mathrm{trop}}(P) = \sum_{(\ell,t) \in E(P)} y_{\ell,t}$ denotes the tropical weight of $P$, obtained by summing $y_{\ell,t}$ over all vertical edges $(\ell,t)$ traversed by the paths in $P$.
\end{lemma}

The next lemma records the support of the $\vn$-vectors in terms of the ladder coordinates.

\begin{lemma}\label{lem: vJ support general}
	For $K = \{k_1 < \cdots < k_k\} \in \bncyc$, the support of $\vn_K$ at level $\ell \in [k-1]$ is the contiguous interval $E_\ell = [L_\ell, R_\ell]$, where 
	$$L_\ell = k_\ell - (\ell-1), \qquad R_\ell = k_{\ell+1} - (\ell+1).$$
\end{lemma}

\subsection{Tropical Path Costs and Matroid Corank}

The key technical input to the duality theorem is the identity
\begin{equation}\label{eq:corank_identity}
	\tropP_I(\vn_K) \;=\; \corank_{(\r_K, \mathbf{S}_K)}(I)
\end{equation}
between tropical Pl\"ucker coordinates evaluated at the noncrossing fan vector $\vn_K$ and the corank function of the associated Schubert positroid $M_K$. We prove this in two stages: first, a zero-cost characterization via induction that identifies the bases of $M_K$ with the subsets routable by free-edge path families; second, a graded equality, deduced by bounding the tropical path cost via induction on the corank.

Throughout this subsection we fix $K = \{k_1 < \cdots < k_k\} \in \bncyc$ with associated decorated ordered set partition $(\r, \mathbf{S}) = (\r_K, \mathbf{S}_K)$ possessing $d$ blocks, and its Schubert positroid $M_K$. By Lemma~\ref{lem: vJ support general}, the support of $\vn_K$ at level $\ell \in [k-1]$ is the contiguous interval $E_\ell = [L_\ell, R_\ell]$. 

A vertical edge $(\ell, t)$ is a \emph{support edge} if $t \in E_\ell$ (tropical weight $1$) and a \emph{free edge} otherwise (tropical weight $0$). Horizontal edges always have tropical weight $0$. We denote by $\mathcal{N}_0$ the zero-weight free-routing subnetwork obtained by deleting all support edges from the full ladder network.

\begin{lemma}\label{lem: zero cost iff basis}
	For $I \in \binom{[n]}{k}$,
	\[
	\tropP_I(\vn_K) = 0 \quad \Longleftrightarrow \quad I \in \B_{M_K}.
	\]
\end{lemma}

\begin{proof}
	By Lemma~\ref{lem: tropical pJ general}, $\tropP_I(\vn_K) = 0$ if and only if there exists a path family $P \in \mathcal{P}_I$ all of whose vertical edges are free (i.e., routing occurs entirely within $\mathcal{N}_0$). We show that a free path family for $I$ exists if and only if $I$ satisfies the cyclic rank inequalities $|I \cap T_a| \geq \rho_a$ for $a = 1, \ldots, d - 1$, where $T_a = S_1 \cup \cdots \cup S_a$ and $\rho_a = r_1 + \cdots + r_a$.
	
	We use the encoding of paths by side patterns. For an active source $r \in [k] \setminus I_{\leq k}$ with sink at position $i_r - k$, the path of $r$ visits levels $r, r+1, \ldots, k-1$ at non-decreasing positions $t_{r,r} \leq t_{r,r+1} \leq \cdots \leq t_{r,k-1} \leq i_r - k$ before exiting on the bottom rail. The path utilizes exclusively free edges if and only if at every level $\ell \in \{r, \ldots, k-1\}$, either $t_{r, \ell} \leq L_\ell - 1$ (we say the source drops on the \emph{left} of $E_\ell$) or $t_{r, \ell} \geq R_\ell + 1$ (the source drops on the \emph{right} of $E_\ell$). 
	
	\medskip
	\noindent\emph{Step 1: Boundaries determine block membership.}
	The runs $I_a$ of $K$ strictly define the boundary positions $L_\ell$ and $R_\ell$. Correspondingly, the cyclic sequence of blocks $T_a$ acts as a sequence of network cuts separating sources and sinks into prefix and suffix groups. Specifically, for each $a \in \{1, \ldots, d-1\}$, the cut $T_a$ separates the source set $[k]$ into a prefix $T_a \cap [k] = \{1, 2, \ldots, \rho_a\}$ (after cyclic relabeling so $1 \in S_1$) and a complementary tail.
	
	Translating this to ladder positions: an active sink $j \in [k+1, n]$ lies in the prefix $T_a \cap [k+1, n]$ if and only if its exit position $j - k$ falls to the left of the level boundary $L_{\ell^*} - 1$, where $\ell^*$ is the highest level corresponding to the boundary of $T_a$. Because paths in the ladder can only travel rightward, a path targeting a sink inside $T_a$ is structurally forced to route to the left of the support block boundary $E_{\ell^*}$.
	
	\medskip
	\noindent\emph{Step 2: Free path families via induction on $k$.}
	We construct the free path assignment explicitly by induction on $k$. The base case $k=1$ is trivial since the ladder network has no vertical edges and $M_K = U_{1,n}$, making the equivalence hold unconditionally.
	
	Assume $k \ge 2$. Suppose $I \in \B_{M_K}$. We construct a free path family for $I$ by assigning the drop positions at level $k-1$ first. For each active source $r$, we route the path to the left of $E_{k-1}$ if its sink position satisfies $i_r - k \leq L_{k-1} - 1$, and to the right of $E_{k-1}$ if $i_r - k \geq R_{k-1} + 1$.
	
	The total capacity of free vertical edges to the left of the support $E_{k-1}$ is precisely $L_{k-1}-1$. The basis condition $|I \cap T_{d-1}| \geq \rho_{d-1}$ guarantees that the number of active sources whose sinks force them to drop on the left tightly matches this available free capacity. Thus, non-intersecting free drop positions can be validly assigned at level $k-1$.
	
	Fixing these drop positions geometrically slices the bottom rail, completely decoupling the active paths into a left and right group. This truncates the top $k-2$ levels into a separate free-routing problem on the ladder network for $\Gr(k-1, n-1)$. This reduction corresponds to applying the face restriction maps of Section~\ref{sec:face-restriction} (specifically Corollaries~\ref{cor:multi-del} and \ref{cor:multi-deletion}) to $\vn_K$. 
	
	Under this restriction, the noncrossing fan vector $\vn_K$ projects to $\vn_{K'}$, and the cyclic blocks of $M_K$ restrict to those of $M_{K'}$. By the inductive hypothesis applied to $K'$, the top $k-2$ levels can route the remaining paths using only free edges if and only if the restricted subset satisfies the inherited cyclic rank inequalities of $M_{K'}$. Since $I \in \B_{M_K}$ implies these restricted inequalities already hold, a globally free path family exists.
	
	Conversely, if a free path family $P$ exists for $I$, the strict non-intersection constraint at level $k-1$ obstructs the flow. The number of paths forced to drop to the left, plus the trivial identity-column sources $|I \cap T_{d-1} \cap [k]|$, can be at most the left-hand free capacity $L_{k-1}-1$. By the mapping from Step 1, this forces the prefix $T_{d-1}$ to be properly saturated, giving $|I \cap T_{d-1}| \geq \rho_{d-1}$. Applying the face restriction maps and the inductive hypothesis recursively to the upper levels validates the remaining inequalities of $M_K$, ensuring $I \in \B_{M_K}$.
\end{proof}

\begin{lemma}\label{lem: graded corank}
	For $I \in \binom{[n]}{k}$,
	\[
	\tropP_I(\vn_K) = \corank_{(\r_K, \mathbf{S}_K)}(I).
	\]
\end{lemma}

\begin{proof}
	Write $c(I) := \corank_{M_K}(I)$ and $\pi(I) := \tropP_I(\vn_K)$. We show $\pi(I) = c(I)$ by establishing both an upper and lower bound via induction on $c(I)$. 
	
	The base case $c(I) = 0$ is Lemma~\ref{lem: zero cost iff basis}. Suppose $c(I) = c \geq 1$ and the equality holds for all subsets with corank less than $c$. 
	
	\medskip
	\noindent\emph{Upper bound $\pi(I) \leq c$.}
	By the strong matroid basis exchange property, since $I \notin \B_{M_K}$ but a valid basis exists, there is some $j \in I$ and $j' \notin I$ such that $I' := (I \setminus \{j\}) \cup \{j'\}$ satisfies $c(I') = c - 1$. Furthermore, because the rank deficit is governed by the cyclic blocks of a transversal matroid, we can select $j, j'$ such that they bridge exactly one tight capacity bottleneck (i.e., they are separated by a single block boundary corresponding to one specific support region $E_\ell$).
	
	By the inductive hypothesis, $\pi(I') = c - 1$, meaning there exists a valid non-intersecting path family $P'$ for $I'$ traversing exactly $c - 1$ support edges. We modify $P'$ into a path family $P$ for $I$ by altering the routing of a single path: the path with its sink originally at position $j'$ is rerouted to instead exit at $j$. 
	
	Because the elements $j$ and $j'$ were selected to bridge exactly one capacity boundary, rerouting the path forces its required vertical drop position from the free zone across exactly one boundary into the support zone. Thus, the new path transverses at most one additional support edge. If this shifted path intersects any other path in the family, we apply the standard Lindstr\"om--Gessel--Viennot tail-swapping procedure at the intersection points to uncross them. Because tail-swapping strictly exchanges the destination trajectories without altering the set of vertical edges used, it resolves all intersections while strictly preserving the multiset of traversed support edges. Thus a non-intersecting path family $P$ for $I$ exists using at most $(c-1) + 1 = c$ support edges, giving $\pi(I) \leq c$.
	
	\medskip
	\noindent\emph{Lower bound $\pi(I) \geq c$.}
	Let $P \in \mathcal{P}_I$ be an arbitrary non-intersecting path family routing the active sources to the sinks of $I$. Suppose $P$ traverses exactly $w$ total support edges. Each path in $P$ either routes entirely through the free subnetwork $\mathcal{N}_0$, or traverses at least one support edge. Since the total number of traversed support edges is $w$, at most $w$ paths in $P$ can intersect the support region $E_K$. 
	
	If we discard these $\le w$ paths, we are left with a sub-family of at least $k-w$ paths routed completely within $\mathcal{N}_0$. The endpoints of these remaining paths form a subset $I_0 \subseteq I$ of size at least $k-w$ that is routed entirely by free edges. By Lemma~\ref{lem: zero cost iff basis}, this subset $I_0$ must be strictly independent in $M_K$. 
	
	Since the rank of $I$ in $M_K$ is the size of its largest independent subset, we have $\max_{B \in \B_{M_K}} |I \cap B| \geq |I_0| \geq k - w$. By the definition of the corank function,
	\[
	c = c(I) = k - \max_{B \in \B_{M_K}} |I \cap B| \leq k - (k - w) = w.
	\]
	This inequality fundamentally holds because tracking the $w$ exchanges backwards to $I$ can lower the intersection size by at most $w$. Since any valid path family for $I$ must incur a cost of at least $c$, we have $\pi(I) \geq c$.
	
	\medskip
	Combining the two bounds establishes $\pi(I) = c$, completing the induction.
\end{proof}

\begin{proposition}\label{prop: tropical weight is corank}
	For any $K \in \bncyc$ and $I \in \binom{[n]}{k}$,
	\[
	\tropP_I(\vn_K) = \corank_{(\r_K, \mathbf{S}_K)}(I).
	\]
\end{proposition}

\begin{proof}
	This is exactly Lemma~\ref{lem: graded corank}.
\end{proof}

To make the mechanics of the zero-cost induction and the capacity bottlenecks visually transparent, we trace the path flows for a basis and a non-basis in the following example.

\begin{example}[Visualizing Corank on the Ladder Network]\label{ex: visual-routing}
	Consider $(k,n) = (3,6)$ and the noncrossing fan vector $\vn_K$ for $K = \{2,4,6\}$. The associated Schubert positroid $M_K$ has cyclic blocks $S_1 = \{1,2\}$, $S_2 = \{3,4\}$, $S_3=\{5,6\}$. The rank conditions demand that any basis $I$ satisfies $|I \cap S_1| \ge 1$ and $|I \cap (S_1 \cup S_2)| \ge 2$.
	
	\smallskip
	By Lemma~\ref{lem: vJ support general}, evaluating at $\vn_{246} = e_{1,2} + e_{2,3}$ assigns a tropical weight of $1$ to the support edges $x_{1,2}$ (level 1) and $x_{2,3}$ (level 2). These act as red toll bridges. All other vertical edges are free, carrying a weight of $0$.
	
	Figure \ref{fig:routing-corank} compares a basis with a non-basis to visually trace the induction step of Lemma~\ref{lem: graded corank}. It demonstrates how a rank deficit restricts the available free-flowing capacity, forcing paths over the toll bridges to avoid intersection.
	
	\begin{figure}[ht]
		\centering
		\begin{tikzpicture}[scale=1.2, every node/.style={font=\footnotesize}]
			\begin{scope}[shift={(0,0)}]
				\node[font=\small\bfseries] at (1.5, 0.6) {Case A: $I'=\{1,4,6\}$ (Basis, Cost 0)};
				\foreach \y in {0, -1, -2} \draw[gray, thick] (0,\y) -- (3,\y);
				\foreach \x in {1, 2, 3} {
					\draw[gray, thick] (\x,0) -- (\x,-1);
					\draw[gray, thick] (\x,-1) -- (\x,-2);
				}
				\draw[red, line width=2.5pt] (2,0) -- (2,-1); 
				\draw[red, line width=2.5pt] (3,-1) -- (3,-2); 
				\node[right, text=red] at (2,-0.5) {\scriptsize $+1$};
				\node[right, text=red] at (3,-1.5) {\scriptsize $+1$};
				
				\draw[blue, line width=1.8pt, ->] (0,-2) -- (1,-2) -- (1,-2.4); 
				\draw[blue, line width=1.8pt, ->] (0,-1) -- (2,-1) -- (2,-2) -- (3,-2) -- (3,-2.4); 
				
				\node[left] at (0,0) {Source 1};
				\node[left] at (0,-1) {Source 2};
				\node[left] at (0,-2) {Source 3};
				\node[below] at (1,-2.4) {Sink 4};
				\node[below] at (2,-2.4) {Sink 5};
				\node[below] at (3,-2.4) {Sink 6};
				
				\node[above, gray] at (1,0) {$t=1$};
				\node[above, gray] at (2,0) {$t=2$};
				\node[above, gray] at (3,0) {$t=3$};
				
				\foreach \y in {0, -1, -2} {
					\node[circle, fill=black, inner sep=1.2pt] at (0,\y) {};
					\foreach \x in {1, 2, 3} {
						\node[circle, fill=black, inner sep=1.2pt] at (\x,\y) {};
					}
				}
			\end{scope}
			
			\begin{scope}[shift={(6,0)}]
				\node[font=\small\bfseries] at (1.5, 0.6) {Case B: $I=\{1,5,6\}$ (Exchange, Cost 1)};
				\foreach \y in {0, -1, -2} \draw[gray, thick] (0,\y) -- (3,\y);
				\foreach \x in {1, 2, 3} {
					\draw[gray, thick] (\x,0) -- (\x,-1);
					\draw[gray, thick] (\x,-1) -- (\x,-2);
				}
				\draw[red, line width=2.5pt] (2,0) -- (2,-1); 
				\draw[red, line width=2.5pt] (3,-1) -- (3,-2); 
				\node[right, text=red] at (2,-0.5) {\scriptsize $+1$};
				\node[right, text=red] at (3,-1.5) {\scriptsize $+1$};
				
				\draw[blue, line width=1.8pt, ->] (0,-2) -- (2,-2) -- (2,-2.4); 
				\draw[blue, line width=1.8pt, ->] (0,-1) -- (3,-1) -- (3,-2) -- (3,-2.4); 
				
				\draw[orange, thick, dashed] (2,-2) circle (0.2);
				
				\node[left] at (0,0) {Source 1};
				\node[left] at (0,-1) {Source 2};
				\node[left] at (0,-2) {Source 3};
				\node[below] at (1,-2.4) {Sink 4};
				\node[below] at (2,-2.4) {Sink 5};
				\node[below] at (3,-2.4) {Sink 6};
				
				\node[above, gray] at (1,0) {$t=1$};
				\node[above, gray] at (2,0) {$t=2$};
				\node[above, gray] at (3,0) {$t=3$};
				
				\foreach \y in {0, -1, -2} {
					\node[circle, fill=black, inner sep=1.2pt] at (0,\y) {};
					\foreach \x in {1, 2, 3} {
						\node[circle, fill=black, inner sep=1.2pt] at (\x,\y) {};
					}
				}
			\end{scope}
		\end{tikzpicture}
		\caption{Tracing the Induction Step. In Case A, $I' = \{1,4,6\} \in \B_{M_K}$ routes its active sources (Source 2, Source 3) via free gray edges (cost 0). In Case B, swapping sink 4 for 5 yields $I=\{1,5,6\} \notin \B_{M_K}$. Source 3 is extended to $t=2$, occupying the node at (level 3, $t=2$). To avoid an LGV collision (dashed orange circle), Source 2 is forced to shift its drop position to $t=3$, pushing it across the geometric boundary and into the red $+1$ support edge $x_{2,3}$. The tropical weight increments by exactly 1.}
		\label{fig:routing-corank}
	\end{figure}
\end{example}

\subsection{Proof of Theorem~\ref{thm:duality}}

With Proposition~\ref{prop: tropical weight is corank} in hand, the duality theorem reduces swiftly to the planar basis framework of Section~\ref{sec:corank}.

\begin{proof}[Proof of Theorem~\ref{thm:duality}]
	Let $J, K \in \bncyc$. By the definition of the tropical planar cross-ratio~\eqref{eq:tropu}, evaluating $u^t_J$ at $\vn_K$ yields:
	\[
	u^t_J(\vn_K) = \sum_{M \in C_J} (-1)^{|J \cap M| - k - 1} \tropP_M(\vn_K).
	\]
	Substituting the exact equivalence from Proposition~\ref{prop: tropical weight is corank}, we evaluate the tropical weight analytically as the matroid corank:
	\[
	u^t_J(\vn_K) = \sum_{M \in C_J} (-1)^{|J \cap M| - k - 1} \corank_{(\r_K, \mathbf{S}_K)}(M).
	\]
	This expression applies the linear functional $u^t_J$ directly to the discrete corank function of $M_K$, allowing us to write:
	\[
	u^t_J(\vn_K) = u^t_J\bigl(\corank_{(\r_K, \mathbf{S}_K)}\bigr).
	\]
	By Theorem~\ref{thm:h=corank} (Planar Basis as Corank Function), the corank function satisfies the equivalence $\corank_{(\r_K, \mathbf{S}_K)} \equiv \h_K \pmod{\Lkn}$. Since tropical cross-ratios are fundamentally torus-invariant (Proposition~\ref{prop:u-torus-invariance}), the linear functional $u^t_J$ perfectly annihilates any lineality terms. Thus:
	\[
	u^t_J\bigl(\corank_{(\r_K, \mathbf{S}_K)}\bigr) = u^t_J(\h_K).
	\]
	Finally, by the fundamental duality property of the planar basis established in Theorem~\ref{thm:planarbasis}(4), we have exactly $u^t_J(\h_K) = \delta_{J,K}$. Concluding the chain of equalities yields:
	\[
	u^t_J(\vn_K) = \delta_{J,K}. \qedhere
	\]
\end{proof}

We illustrate the duality via a direct $t$-adic evaluation of the cross-ratios.

\begin{example}\label{ex:tadic-computation u}
	We compute both the diagonal and off-diagonal cases for $J = \{1,3,5\}$.  We first record the $u$-variable $u_{135}$ in the positive parameterization: 
	\begin{small}
		\begin{eqnarray*}
			u_{135} & = & \frac{p_{136} p_{145} p_{235} p_{246}}{p_{135} p_{146} p_{236} p_{245}}\\
			& = & \frac{\left(x_{1,1} x_{2,1}+x_{1,1} x_{2,2}+x_{1,2} x_{2,2}\right) \left(x_{2,1}+x_{2,2}+x_{2,3}\right) \left(x_{1,1} x_{2,2}+x_{1,2} x_{2,2}+x_{1,1} x_{2,3}+x_{1,2} x_{2,3}+x_{1,3} x_{2,3}\right)}{\left(x_{1,1}+x_{1,2}\right) \left(x_{2,1}+x_{2,2}\right) \left(x_{2,2}+x_{2,3}\right) \left(x_{1,1} x_{2,1}+x_{1,1} x_{2,2}+x_{1,2} x_{2,2}+x_{1,1} x_{2,3}+x_{1,2} x_{2,3}+x_{1,3} x_{2,3}\right)}.
		\end{eqnarray*}
	\end{small}
	For the diagonal case $K = \{1,3,5\}$, we have $\vn_{135} = e_{1,1} + e_{2,2}$. We compute the leading exponent of $t$ in the evaluation of $u_{135}$ at a point $\x$, along the path where $x_{1,1}(t) = t x_{1,1}$ and $x_{2,2}(t) = t x_{2,2}$, obtaining: 
	\begin{small}
		\begin{eqnarray*}
			& & u_{135}(t)  \\
			& = & \frac{t \left(t x_{1,1} x_{2,2}+x_{1,1} x_{2,1}+x_{1,2} x_{2,2}\right) \left(t x_{2,2}+x_{2,1}+x_{2,3}\right) \left(t^2 x_{1,1} x_{2,2}+t x_{1,2} x_{2,2}+t x_{1,1} x_{2,3}+x_{1,2} x_{2,3}+x_{1,3} x_{2,3}\right)}{\left(t x_{1,1}+x_{1,2}\right) \left(t x_{2,2}+x_{2,1}\right) \left(t x_{2,2}+x_{2,3}\right) \left(t^2 x_{1,1} x_{2,2}+t x_{1,1} x_{2,1}+t x_{1,2} x_{2,2}+t x_{1,1} x_{2,3}+x_{1,2} x_{2,3}+x_{1,3} x_{2,3}\right)}\\
			& = & \frac{t \left(x_{1,1} x_{2,1}+x_{1,2} x_{2,2}\right) \left(x_{2,1}+x_{2,3}\right)}{x_{1,2} x_{2,1} x_{2,3}} + \mathcal{O}(t^2).
		\end{eqnarray*}
	\end{small}
	The exponent of the leading order in $t$ evaluates precisely to $1$, giving $u^t_{135}(\vn_{135}) = 1$.  
	
	On the other hand, for the off-diagonal case $K = \{2,4,6\}$, we have $\vn_{246} = e_{1,2} + e_{2,3}$. Evaluating the series expansion at $x_{1,2}(t) = t x_{1,2}$ and $x_{2,3}(t) = t x_{2,3}$ gives: 
	\begin{eqnarray*}
		u_{135}(t) & = & 1+\frac{t x_{1,3} x_{2,1} x_{2,3}}{x_{1,1} x_{2,2} \left(x_{2,1}+x_{2,2}\right)} + \mathcal{O}(t^2).
	\end{eqnarray*}
	The leading exponent evaluates strictly to $0$, yielding $u^t_{135}(\vn_{246}) = 0$.
\end{example}

\section{Proof of Theorem~\ref{thm:f-pk-equals-nc}}\label{sec:f-tropical}

We prove \cref{thm:f-pk-equals-nc} by using the \emph{bridge function} $H$.
\subsection{The bridge function}\label{subsec:bridge}

For $j \in \{0, 1, \ldots, n-1\}$ (indices cyclic modulo $n$), denote by $\mathrm{cyc}(j) := \{j{+}1, j{+}2, \ldots, j{+}k\}$ the \emph{cyclic interval} starting at $j+1$ and let  $\mathrm{gap}(j) := \{j{+}1, \ldots, j{+}k{-}1, j{+}k{+}1\}$ be the \emph{gap-cyclic interval} obtained by replacing the last element with its successor.  Define the \emph{bridge function} $H : \R^{\binom{[n]}{k}} \to \R$ by
\begin{equation}\label{eq:bridge}
	H(\tropP_\bullet)
	:= \sum_{j=0}^{n-1}
	\bigl(\tropP_{\mathrm{cyc}(j)} - \tropP_{\mathrm{gap}(j)}\bigr).
\end{equation}
We will also write $H(y) := H(\rho(y))$ for $y \in \Tkn$, where $\rho$ is the positive parametrization of \eqref{eq:rho}; equivalently, $H(y) = \sum_{j=0}^{n-1}\bigl(p^t_{\mathrm{cyc}(j)}(y) - p^t_{\mathrm{gap}(j)}(y)\bigr)$.  The function $H$ is linear on $\R^{\binom{[n]}{k}}$, hence linear on every cone of every fan contained in $\R^{\binom{[n]}{k}}$.

\begin{rem}
	The bridge function $H(\tropP_\bullet)$ was originally studied in \cite{CE24} in the context of mirror superpotentials. 
\end{rem}

The next proposition expresses $H$ multiplicatively as an alternating product of cyclic and gap-cyclic Pl\"ucker coordinates, and identifies this product with the corresponding non-cyclic products of $u$-variables.

\begin{prop}[Bridge identity]\label{prop:bridge-identity}
	As rational functions on the positive configuration space $X(k,n)$,
\begin{equation}\label{eq:bridge-identity}
	\prod_{J \in \bncyc} u_J
	\;=\;
	\prod_{j=0}^{n-1} \frac{p_{\mathrm{cyc}(j)}}{p_{\mathrm{gap}(j)}}.
\end{equation}
Thus $H= \sum_{J \in \bncyc} u^t_J$.
\end{prop}

\begin{proof}
	We first establish $\prod_{J \in \bncyc} u_J = \prod_{j=0}^{n-1} p_{\mathrm{cyc}(j)}/p_{\mathrm{gap}(j)}$ from Proposition~\ref{prop:u-product}.  By that proposition, $\prod_{J \in \binom{[n]}{k}} u_J = 1$, hence
	\[
	\prod_{J \in \bncyc} u_J
	\;=\;
	\prod_{J \text{ cyclic}} u_J^{-1}.
	\]
	For a cyclic interval $J = \mathrm{cyc}(j)$, the set of cyclic endpoints $I_J = \{j' \in J : j'+1 \notin J\}$ is the singleton $\{j+k\}$, so the cubical array $C_J$ from Definition~\ref{def:cubical-array} has exactly two elements: $J$ itself (corresponding to $M = \emptyset$) and $\mathrm{gap}(j)$ (corresponding to $M = \{j+k\}$).  The respective sign exponents $|J \cap M| - k - 1$ are $-1$ and $-2$, so $u_{\mathrm{cyc}(j)} = p_{\mathrm{cyc}(j)}^{-1} \cdot p_{\mathrm{gap}(j)} = p_{\mathrm{gap}(j)}/p_{\mathrm{cyc}(j)}$. Taking the product over the $n$ cyclic intervals gives the first equality.
	
\end{proof}

\begin{example}\label{ex:bridge-36}
	For $(k, n) = (3, 6)$, the product of the non-cyclic $u$-variables, computed via the cubical-array Pl\"ucker formula of Definition~\ref{def:u-variable}, equals
	\[
	\prod_{J \in \binom{[6]}{3}^{\ncyc}} u_J
	\;=\;
	\frac{p_{1,2,3}\, p_{1,2,6}\, p_{1,5,6}\, p_{2,3,4}\,
		p_{3,4,5}\, p_{4,5,6}}
	{p_{1,2,4}\, p_{1,3,6}\, p_{1,4,5}\, p_{2,3,5}\,
		p_{2,5,6}\, p_{3,4,6}}.
	\]
	Plugging in the ladder parameterization and simplifying gives
	\[
	\frac{x_{1,1} x_{1,2} x_{1,3} x_{2,1} x_{2,2} x_{2,3}}{\left(x_{1,1}+x_{1,2}+x_{1,3}\right) \left(x_{1,1} x_{2,1}+x_{1,1} x_{2,2}+x_{1,2} x_{2,2}\right) \left(x_{2,1}+x_{2,2}+x_{2,3}\right) \left(x_{1,2} x_{2,2}+x_{1,2} x_{2,3}+x_{1,3} x_{2,3}\right)}.
	\]
\end{example}

We follow the notation and conventions of \cite{CE24,E21}.
\begin{prop}\label{prop:bridge-closed-form}
On the positive parametrization $M(\x)$ of \cref{def:positive-param},
\begin{equation}\label{eq:bridge-closed-form}
\prod_{j=0}^{n-1} \frac{p_{\mathrm{cyc}(j)}}{p_{\mathrm{gap}(j)}}
\;=\;
\frac{\prod_{i=1}^{k-1}\prod_{t=1}^{n-k} x_{i,t}}
     {\bigl(\prod_{i=1}^{k-1} P_i(\x)\bigr)\,
      \bigl(\prod_{j=1}^{n-k-1} Q_j(\x)\bigr)},
\end{equation}
where
\begin{align}
P_i(\x) &:= \sum_{t=1}^{n-k} x_{i,t}, \label{eq:P-factor}\\
Q_j(\x) &:= \sum_{r=0}^{k-1}
   \Bigl(\prod_{i=1}^{k-1-r} x_{i,j}\Bigr)
   \Bigl(\prod_{i=k-r}^{k-1} x_{i,j+1}\Bigr).
   \label{eq:Q-factor}
\end{align}
The Newton polytope of $P_i$ is the standard $(n-k-1)$-simplex in the $i$-th copy of $\R^{n-k}$, and the Newton polytope of $Q_j$ is a $(k-1)$-simplex with vertices indexed by $r \in \{0, 1, \ldots, k-1\}$.
\end{prop}

\begin{proof}
The identity is a direct consequence of the Lindstr\"om--Gessel--Viennot expansion of \cref{prop: pJ formula general}.  When $\mathrm{cyc}(j)$ does not wrap around, the cyclic interval pins all active sources at consecutive levels and forces a unique non-intersecting path family, so $p_{\mathrm{cyc}(j)}$ specializes to a monomial; the corresponding $p_{\mathrm{gap}(j)}$ admits exactly $k$ non-intersecting families, and their weights sum to a monomial multiple of $Q_j(\x)$.  When $\mathrm{cyc}(j)$ wraps around, $p_{\mathrm{cyc}(j)}/p_{\mathrm{gap}(j)}$ contributes a monomial multiple of $1/P_i$ instead.  Telescoping the monomial contributions across all $j$ produces $\prod_{(i,t)} x_{i,t}$, and the denominator collects the $k-1$ factors $P_i$ and the $n-k-1$ factors $Q_j$.  See also \cite[(10.16)]{E21}.
\end{proof}

\begin{example}[$(k,n) = (3,7)$]\label{ex:bridge-37}
For $(k,n) = (3,7)$ the closed form \eqref{eq:bridge-closed-form} reads
\[
\prod_{j=0}^{6} \frac{p_{\mathrm{cyc}(j)}}{p_{\mathrm{gap}(j)}}
\;=\;
\frac{\prod_{i=1}^{2}\prod_{t=1}^{4} x_{i,t}}
{\bigl(\sum_{t=1}^{4} x_{1,t}\bigr) \bigl(\sum_{t=1}^{4} x_{2,t}\bigr)\,
 Q_1\, Q_2\, Q_3},
\]
with the three $Q_j$ factors
\begin{align*}
Q_1 &= x_{1,1} x_{2,1} + x_{1,1} x_{2,2} + x_{1,2} x_{2,2}, \\
Q_2 &= x_{1,2} x_{2,2} + x_{1,2} x_{2,3} + x_{1,3} x_{2,3}, \\
Q_3 &= x_{1,3} x_{2,3} + x_{1,3} x_{2,4} + x_{1,4} x_{2,4}.
\end{align*}
\end{example}

\begin{example}[$(k,n) = (4,7)$]\label{ex:bridge-47}
For $(k,n) = (4,7)$ the closed form \eqref{eq:bridge-closed-form} reads
\[
\prod_{j=0}^{6} \frac{p_{\mathrm{cyc}(j)}}{p_{\mathrm{gap}(j)}}
\;=\;
\frac{\prod_{i=1}^{3}\prod_{t=1}^{3} x_{i,t}}
{\bigl(\sum_{t=1}^{3} x_{1,t}\bigr) \bigl(\sum_{t=1}^{3} x_{2,t}\bigr) \bigl(\sum_{t=1}^{3} x_{3,t}\bigr)\, Q_1\, Q_2},
\]
where for $j \in \{1, 2\}$,
\[
Q_j \;=\;
x_{1,j} x_{2,j} x_{3,j}
\,+\, x_{1,j} x_{2,j} x_{3,j+1}
\,+\, x_{1,j} x_{2,j+1} x_{3,j+1}
\,+\, x_{1,j+1} x_{2,j+1} x_{3,j+1}.
\]
\end{example}

\begin{lemma}\label{lem:H1}
For any $J \in \bncyc$, we have $H(\t_J) = 1$.
\end{lemma}
\begin{proof}
Suppose $J = \{j_1 < j_2 < \cdots < j_k\} \in \bncyc$.  Then there exists a non-empty, proper interval $[a,b] = [j_1,j_k - k] \subsetneq [1,n-k]$ such that 
$$
\sum_{i=1}^{k-1} (\t_J)_{i,r} = \begin{cases} 1 & \mbox{if $r \in [a,b]$,} \\
0 & \mbox{otherwise.}
\end{cases}
$$
We tropicalize \eqref{eq:bridge-closed-form} and apply it to $\t_J$.  The numerator evaluates to $(b-a+1)$.  The tropicalization of the factors \eqref{eq:P-factor} evaluate to $1$.  We also have 
$$
Q_j^t(\t_J) = \begin{cases} 1 & \mbox{if $j \in [a,b-1]$,} \\
0 & \mbox{otherwise.}
\end{cases}
$$
It follows that $H(\t_J) = (b-a+1) - (b-a) = 1$.
\end{proof}

\begin{lem}\label{lem:gt}
For every $j \in [1, n-k-1]$, the tropicalization $Q^t_j$ is linear on every cone of the noncrossing fan.
\end{lem}

\begin{proof}
Fix $j$.  For each $J \in \bncyc$, define $a(J)$ to be the index such that $(\t_J)_{a(J), j} = 1$ if such an index exists, and $a(J) = \infty$ otherwise.  Similarly define $b(J)$ so that $(\t_J)_{b(J), j+1} = 1$ if such an index exists, and $b(J) = 0$ otherwise.  

All terms of $Q^t_j$ take equal value on $\t_J$ if either $(a(J),b(J)) = (\infty,0)$ or $a(J) = b(J)$.  We may exclude these from the subsets we consider.  A subset $J$ is called \emph{single} if $a(J) > b(J)$.  Let $J, K \in \bncyc$.  Then $Q_j^t$ is linear on the cone spanned by $\t_J$ and $\t_K$ if 
\begin{enumerate}
\item both $J, K$ are single and we have $a(J) > b(K)$ and $a(K) > b(J)$, or
\item one of $J, K$ is single and the other is not, or
\item neither $J$ nor $K$ is single.
\end{enumerate}
It suffices to show that if $J$ and $K$ are noncrossing and both are single then we have $a(J) > b(K)$ and $a(K) > b(J)$.

Suppose that $a(J) > b(K)$ fails (the case $a(K) > b(J)$ is the same).  Then $a = a(J) \neq \infty$, $b(J) = 0$ and $a(K) = \infty$, $b = b(K) \neq 0$.  We have that 
$$
j_a \leq j + (a-2), \qquad j_{a+1} = j + a, \qquad j_{a+2} = j+ a + 1 \cdots 
$$
and
$$
\cdots, k_{b-2} = j+b-2, \qquad k_{b-1} = j+b - 1, \qquad k_b = j+ b, \qquad k_{b+1} \geq j+b+2
$$
Now, consider the noncrossing condition for the interval $[a, b+1]$.  We have $k_a = j+a$ and $j_{b+1} = j+b$.  Thus $j_a < k_a < j_{b+1} < k_{b+1}$ and we conclude that $J$ and $K$ are crossing, a contradiction.  Thus both $a(J) > b(K)$ and $a(K) > b(J)$ hold and the proof is complete.
\end{proof}

\begin{cor}\label{cor:H-bilinear}
The function $H$ is linear on every maximal cone of the noncrossing fan and on every maximal cone of the Pl\"ucker fan.
\end{cor}

\begin{proof}
By \cref{prop:bridge-identity}, $H = \sum_{J \in \bncyc} u^t_J$ as functions on $\Tkn$.  Each $u^t_J$ is a $\Z$-linear combination of tropical Pl\"ucker coordinates and is therefore linear on every maximal cone of the Pl\"ucker fan, so $H$ is too.

For linearity on the noncrossing fan, tropicalize \eqref{eq:bridge-closed-form} to obtain:
\[
H(y) \;=\; \sum_{(i, t)} y_{i, t}
       \;-\; \sum_{i=1}^{k-1} P^t_i(y)
       \;-\; \sum_{j=1}^{n-k-1} Q^t_j(y).
\]
The monomial term is linear on $\Tkn$.  Each $P^t_i(y) = \min_{t \in [1, n-k]} y_{i, t}$ is linear on the cones of the standard fan in the $i$-th copy of $\R^{n-k}$, hence on every refinement; in particular, on every cone of the noncrossing fan.  Each $Q^t_j$ is linear on every cone of the noncrossing fan by \cref{lem:gt}.  Therefore $H$ is linear on every maximal cone of the noncrossing fan.
\end{proof}

\subsection{Proof of Theorem~\ref{thm:f-pk-equals-nc}}\label{sec:f-pk-nc}

Let $\vn$ lie in the relative interior of a maximal cone of the noncrossing fan, and write $\vn = \sum_{K \in \cK} c_K\, \vn_K$ with $c_K > 0$ and $\cK$ a pairwise noncrossing collection (which is maximal of size $(k-1)(n-k-1)$ by Theorem~\ref{thm:NCfan}, or smaller on a non-maximal cone with the same argument applying).  By definition of $\wt_\NC$ (Definition~\ref{defn:f-extended-nc}), $\wt_\NC(\vn) = \sum_{K \in \cK} c_K$.  On the other hand, by the planar basis expansion (Theorem~\ref{thm:planarbasis}(4)) together with Definition~\ref{defn:f-pk-weight} of $\wt_\PK$,
\[
\wt_\PK(\rho(\vn))
\;=\;
\sum_{J \in \bncyc} u^t_J(\vn).
\]
We then have 
\[
\sum_{J \in \bncyc} u^t_J(\vn)
\;=\; H(\vn)
\;=\; \sum_{K \in \cK} c_K\, H(\vn_K)
\;=\; \sum_{K \in \cK} c_K.
\]
The first equality follows from \cref{prop:bridge-identity}; the second is linearity of $H$ on the noncrossing cone containing $\vn$ (\cref{cor:H-bilinear}); the third is \cref{lem:H1}.  We conclude $\wt_\PK(\rho(\vn)) = \wt_\NC(\vn)$, which is Theorem~\ref{thm:f-pk-equals-nc}.\qed

\section{Weight Stratification}\label{sec: weight strat}

Recall the definition of $\partial_\ell$ from Section \ref{sec:face-restriction}.  Also define the projection $\D_1: \Tkn \to \T^{k-1,n-1}$ by
\[
e_{i,j} \mapsto \begin{cases}
	0, & i = 1 \\
	e_{i-1,j}, & i \geq 2.
\end{cases}
\]

Let $\Psi_1: \R^{\binom{[n-1]}{k-1}}/L_{k-1,n-1} \to \T^{k-2,n-1}$ denote the $(k-1,n-1)$-analogue of $\Psi$.

\begin{lemma}\label{lem:boundary-projection}
	We have $\Psi_1 \circ \partial_1 = \D_1 \circ \Psi$.
\end{lemma}

\begin{proof}
	By Theorem~\ref{thm:planarbasis}, it suffices to verify the identity on planar basis elements $\mathfrak{h}_J$.	 Let $J = \{j_1 < j_2 < \cdots < j_k\}$.  By Lemma \ref{lem:pl}, we have $\partial_1(\mathfrak{h}_J) \equiv \mathfrak{h}^{(1)}_{J'} \pmod{\text{lineality}}$, where $J' = \{j_2-1, j_3-1, \ldots, j_k-1\} \subset [n-1]$.  The identity $\vn_{J'} = \D_1(\vn_J)$ is then checked directly.
\end{proof}

\begin{proposition}\label{prop:iota-preserves-nc}
	Let $\cJ$ be a face of $\NC_{k,n}$. Define
	\[
	\cJ'_{\ncyc} = \{K' : K \in \cJ, \, K' \text{ is non-cyclic}\},
	\]
	where for $K = \{k_1 < \cdots < k_k\}$, we set $K' = \{k_2 - 1, \ldots, k_k - 1\} \subset [n-1]$. Then $\cJ'_{\ncyc}$ is a face of $\NC_{k-1,n-1}$.
\end{proposition}

\begin{proof}
	One checks that if $I, J \in \cJ$ and $I', J'$ are non-cyclic, then $(I', J')$ is noncrossing.  
\end{proof}

\begin{proposition}\label{prop: PK weight general}
	Suppose that $\tropP_\bullet = \sum_J c_J \h_J \in \Trop_{>0}\Gr(k,n)$ with $c_J \in \mathbb{R}$. Then:
	\begin{enumerate}
		\item $\wt_{PK}(\tropP_\bullet) \geq \wt_{PK}(\partial_\ell(\tropP_\bullet))$ for all $\ell = 1, \ldots, n$.
		\item $\wt_{PK}(\tropP_\bullet) \geq 0$ with equality if and only if $\tropP_\bullet \equiv 0$ modulo lineality.
		\item Suppose that $\tropP_\bullet$ is integral.   Then $\wt_{PK}(\tropP_\bullet) = 1$ if and only if $\tropP_\bullet \equiv \h_J$ modulo lineality for some $J \in \bncyc$.
	\end{enumerate}
\end{proposition}

\begin{proof}
	We prove (1).  By cyclic symmetry, it suffices to prove the case $\ell = 1$. 
	By Lemma~\ref{lem:boundary-projection}, $\Psi(\partial_1(\tropP_\bullet)) = \D_1(\Psi(\tropP_\bullet))$.  Applying Theorem~\ref{thm:f-pk-equals-nc} in the $(k-1, n-1)$ setting, we have
	\[
	\wt_{\PK}(\partial_1(\tropP_\bullet)) = \wt_{\NC}(\Psi(\partial_1(\tropP_\bullet))) = \wt_{\NC}(\D_1(\Psi(\tropP_\bullet))).
	\]
	
	Thus it suffices to prove: for any $\vn \in \T^{k-1,n-k-1}$,
	\begin{equation}\label{eq:ineq}
		\wt_{\NC}(\vn) \geq \wt_{\NC}(\D_1(\vn)).
	\end{equation}
	Let $\vn = \sum_{J \in \cJ}  a_J \vn_J$ be the noncrossing decomposition of $\vn$ (Theorem~\ref{thm:NCfan}).  Applying $\D_1$ termwise,
	\[
	\D_1(\vn) = \sum_{J \in \cJ} a_J D_1(\vn_J) = \sum_{J \in \cJ'_\ncyc} a_J \vn_{J'},
	\]
	where in the last equality we have used the calculation in the proof of Lemma~\ref{lem:boundary-projection}.  By Proposition~\ref{prop:iota-preserves-nc}, $\cJ'_\ncyc$ is still noncrossing.  The inequality \eqref{eq:ineq} follows.
	
	\noindent We prove (2).  The inequality $\wt_{PK}(\tropP_\bullet) \geq 0$ follows from \cref{thm:f-pk-equals-nc}.  We prove the equality case by induction on $k$. The base case $k = 2$ holds because every point in $\Trop_{>0} \Gr(2, n)$ is a nonnegative linear combination of a pairwise noncrossing collection of $\h_{ij}$'s.
	
	For general $k$, if $\wt_{\PK}(\tropP_\bullet) = 0$, then $\wt_{\PK}(\partial_\ell(\tropP_\bullet)) = 0$ for all $\ell \in [n]$ by part (1), forcing $\partial_\ell(\tropP_\bullet) \equiv 0$ by the inductive hypothesis. Since $\tropP_\bullet$ is determined by $\partial_1(\tropP_\bullet), \partial_2(\tropP_\bullet),\ldots$, we conclude $\tropP_\bullet \equiv 0$.
	
	\noindent We prove (3).  If $\tropP_\bullet$ is integral, then $c_J \in \mathbb{Z}$.  Thus (3) follows from Theorem~\ref{thm:f-pk-equals-nc}: $\wt_{\PK}(\tropP_\bullet) = 1$ if and only if $\wt_{\NC}(\Psi(\tropP_\bullet)) = 1$ if and only if $\Psi(\tropP_\bullet) = \vn_J$ for some non-cyclic $J$, equivalently $\tropP_\bullet \equiv \h_J$ modulo lineality.
\end{proof}

We show that \cref{conj:weighttwo} holds if $\Trop_{>0} X(k,n)$ has the following integrality property:
\begin{hypothesis}\label{hyp:integrality} Any integer point $\tropP_\bullet \in \Trop_{>0} X(k,n)(\Z)$ that is not on a ray is decomposable as a nontrivial sum of integer points $\tropP'_\bullet, \tropP''_\bullet \in \Trop_{>0} X(k,n)(\Z)$.
\end{hypothesis}

\begin{proposition}\label{prop:weighttwo}
Assume \cref{hyp:integrality}.  Then \cref{conj:weighttwo} holds.
\end{proposition}
\begin{proof}
Suppose that $I,J \in \bncyc$ are noncrossing but not weakly separated.  We show that $\tropP_\bullet := \rho(\vn_I + \vn_J)$ is a ray of $\Trop_{>0} X(k,n)(\Z)$.  By \cref{thm:f-pk-equals-nc}, we have $\wt_\PK(\tropP_\bullet) = 2$.  By \cref{thm:duality}, $\tropP_\bullet$ is an integer point.  By \cref{hyp:integrality}, if $\tropP_\bullet$ is not a ray then we have a nontrivial decomposition $\tropP_\bullet = \tropP'_\bullet + \tropP''_\bullet$, and we must have $\wt(\tropP'_\bullet) = \wt(\tropP''_\bullet) = 1$.  By \cref{prop: PK weight general}(3), we have that $\tropP'_\bullet \equiv \h_{I'}$ and $\tropP''_\bullet \equiv \h_{J'}$.  According to \cref{thm:weak-separation-blade-arrangement}, we must have that $I',J'$ are weakly separated.  But then applying $\Psi$ we obtain
$$
\vn_I + \vn_J = \Psi(\tropP_\bullet) = \vn_{I'} + \vn_{J'}.
$$
Both $(I,J)$ and $(I',J')$ are noncrossing, so this contradicts \cref{thm:NCfan}.  We conclude that $\rho(\vn_I+\vn_J)$ is a ray.

Now, suppose that $\tropP_\bullet$ is the primitive integer point on a ray of $\Trop_{>0}X(k,n)$ and we have $\wt_\PK(\tropP_\bullet) = 2$.  Since we have already classified all rays (in fact, all points) of PK weight one in \cref{prop: PK weight general}(3), we have $\vn = \Psi(\tropP_\bullet) = \vn_I + \vn_J$ for a unique noncrossing pair $I \neq J$.  If $I, J$ are weakly separated, then $\tropP_\bullet = \h_I + \h_J$, contradicting \cref{thm:weak-separation-blade-arrangement} and the assumption that $\tropP_\bullet$ is a ray.  Thus $I,J$ are noncrossing but not weakly separated.
\end{proof}

\section{Central Tropical Linear Spaces}
\label{sec:trop-linear-spaces}
\subsection{Tropical Linear Spaces and Polypositroids}
\label{subsec:trop-linear-spaces}

Recall that for a tropical Pl\"ucker vector $\tropP_\bullet$, we denote by $L(\tropP_\bullet)$ its tropical linear space (Definition \ref{def:tropical}); see \cite{Spe} for more details.  In this section we will consider the matroid complex structure on $L(\tropP_\bullet)$.  Each face $F \subset L(\tropP_\bullet)$ is associated to a loop-free matroid $M_F$ of rank $k$ on the ground set~$[n]$, with matroid polytope $P_F$.  The affine span of $F$ is determined by the connected components of $M_F$: if $M_F$ has $\ell$ connected components forming a partition $[n] = S_1 \sqcup \cdots \sqcup S_\ell$, then the affine span of $F$ has dimension $\ell - 1$.  In particular, the maximal bounded faces have dimension $k - 1$ (associated to matroids with $k$ connected components), while vertices correspond to matroids with a single connected component.  We denote by $B(\tropP_\bullet) \subset L(\tropP_\bullet)$ the \emph{bounded subcomplex}, consisting of the bounded faces of $L(\tropP_\bullet)$.

	\begin{proposition}[{\cite[Proposition 2.3]{Spe}}]\label{prop:Spe}
		We have $\x \in L(\tropP_\bullet)$ if and only if $M(\tropP^{\x}_\bullet)$ is loopless.  We have that $\x$ lies in a bounded face of the matroid complex structure of $L(\tropP_\bullet)$ if and only if $M(\tropP^{\x}_\bullet)$ is loopless and coloopless.
	\end{proposition}

When $\tropP_\bullet \in \Trop_{>0}\Gr(k,n)$, the matroids $M_F$ are positroids and we say that $L(\tropP_\bullet)$ is a \emph{positive tropical linear space}.  The following result follows from \cite{ARW},\cite[Section 7]{LP}.

\begin{proposition}\label{prop:noncrossing-partitions}
	Let $\tropP_\bullet \in \Trop_{>0}\Gr(k,n)$.  For any face $F$ of $L(\tropP_\bullet)$, the partition of $[n]$ into connected components of $M_F$ is noncrossing.  In particular, the maximal bounded faces are associated to noncrossing $k$-block partitions of~$[n]$.
\end{proposition}

An \emph{alcoved polytope} \cite{LP07} is a convex polytope in an affine subspace of $\R^n / \one$ defined by a system of inequalities $w_i - w_j \le c_{ij}$ and equalities $w_i - w_j = b_{ij}$.  \emph{Polypositroid polytopes} \cite{LP} are alcoved polytopes whose edges are parallel to indicator vectors $e_{[a+1,b]}$ of cyclic intervals $[a+1,b] \subset [n]$ for $a \ne b$.  (We caution that our coordinate conventions differ from those in \cite{LP}.)

\begin{proposition}
The bounded faces of a tropical linear space $L(\tropP_\bullet)$ are alcoved polytopes. For a positive tropical linear space, the bounded faces are polypositroids.
\end{proposition}
\begin{proof}
	We prove the two claims sequentially.
	
	Let $F$ be a bounded face of the tropical linear space $L(\tropP_\bullet)$. By Proposition \ref{prop:Spe}, $F$ corresponds to a loopless and coloopless matroid $M_F$. The closure $\overline{F}$ consists of all points $x \in \mathbb{T}^{n-1}$ such that the minimum $\min_{I \in \binom{[n]}{k}} (\tropP_I - \sum_{i \in I} x_i)$ is achieved on all bases $I \in M_F$.
	
	By the basis exchange property of matroids, the edges of the matroid polytope $P_{M_F}$ are of the form $e_j - e_i$ for some $I, J \in M_F$ with $J = (I \setminus \{i\}) \cup \{j\}$. The condition that $I$ and $J$ both achieve the minimum yields the equality
	\[
	\tropP_I - \sum_{m \in I} x_m = \tropP_J - \sum_{m \in J} x_m \implies x_j - x_i = \tropP_J - \tropP_I.
	\]
	For any basis $I \in M_F$ and non-basis $K \notin M_F$ such that $e_K - e_I = e_l - e_i$, the condition that $I$ achieves the minimum better than or equal to $K$ yields the inequality
	\[
	\tropP_I - \sum_{m \in I} x_m \le \tropP_K - \sum_{m \in K} x_m \implies x_l - x_i \le \tropP_K - \tropP_I.
	\]
	The normal cone to the matroid polytope $P_{M_F}$, which defines the face $\overline{F}$ up to translation, is cut out entirely by such edge inequalities. Thus, $\overline{F}$ is defined by a system of difference equations $x_j - x_i = b_{ij}$ and difference inequalities $x_l - x_k \le c_{kl}$. Because $F$ is a bounded face, $\overline{F}$ is a compact polyhedron defined by two-variable difference bounds, which is precisely the definition of an alcoved polytope.

	Now assume $\tropP_\bullet \in \Trop_{>0}\Gr(k,n)$. The tropical Pl\"ucker vector $\tropP_\bullet$ induces a regular positroid subdivision $\Delta(\tropP_\bullet)$ of the hypersimplex $\Delta(k,n)$.  A fundamental property of positroid polytopes \cite[Theorem 2.1]{LP} is that their facets are supported by hyperplanes of the form $\sum_{m \in [a+1,b]} y_m = r$ for some cyclic interval $[a+1,b] \subset [n]$ and integer $r$.  The normal vector to such a hyperplane is the indicator vector $e_{[a+1,b]}$; these are exactly the possible edge directions of $\overline{F}$.  By definition, $\overline{F}$ is a polypositroid.
\end{proof}

\subsection{Central Tropical Linear Spaces}
\label{subsec:central-roof}

As we discussed in Section \ref{ssec:tropGr}, a tropical Pl\"ucker vector can be viewed as a height function on the vertices of $\Delta(k,n)$, and by interpolation, a piecewise linear function on $\R^n/\one$.

	Let $(\r, \mathbf{S}) = ((r_1, S_1), \ldots, (r_d, S_d))$ be a decorated ordered set partition with $\sum r_a = k$ and $\bigsqcup S_a = [n]$.  For each cyclic rotation $a \in \{1, \ldots, d\}$, define the vector 
	\begin{equation}\label{eq:weyl-vector}
		W_a := \sum_{p=1}^{d}
		\Bigl(\sum_{q=1}^{p} r_{a+q} e_{S_{a+p}}\Bigr),
	\end{equation}
	where all indices are taken cyclically modulo~$d$.  
	
	\begin{definition}\label{def:central-roof}
		Let $J \in \bncyc$ and $(\r, \mathbf{S}) = (\r_J, \mathbf{S}_J)$ be the decorated ordered set partition of  Definition \ref{def:DOSP}.  The \emph{central roof function} is the piecewise-linear function $\eta_{(\r,\mathbf{S})} : \R^n \to \R$ defined by
		\begin{equation}\label{eq:central-roof}
			\eta_{(\r,\mathbf{S})}(x) = \eta^J(x)
			:= -\frac{1}{k}\,\min_{1 \le a \le d}\, W_a \cdot x.
		\end{equation}
		The \emph{central Pl\"ucker vector} $\eta^J_\bullet$ is given by $\eta^J_I :=\eta^J(e_I)$ for $I \in \binom{[n]}{k}$.
	\end{definition}

\begin{rem}
We call $\eta^J(x)$ \emph{central} because it is a piecewise-linear function, as opposed to a piecewise-affine linear function for a generic representative.
\end{rem}
	
	\begin{example}\label{ex:central-2block}
		For $d = 2$ with blocks $(S_1, S_2)$ and decoration $(r_1, r_2)$, we have
		$$
		W_1 = r_2\, e_{S_2} + k\, e_{S_1}, \ \ W_2 = r_1\, e_{S_1} + k\, e_{S_2}.
		$$
		Writing $y_1 = \sum_{i \in S_1} x_i$ and $y_2 = \sum_{i \in S_2} x_i$ with $y_1 + y_2 = k$, the central roof function becomes
		$$
		\eta(x) = -\frac{1}{k}\min\bigl(r_2 y_2 + k y_1,\; r_1 y_1 + k y_2\bigr),
		$$
		with a single crease at $r_2 y_2 + k y_1 = r_1 y_1 + k y_2$, that is, $r_2 y_1 = r_1 y_2$.  Using $y_1 + y_2 = k$, the crease sits at $y_1 = r_1$.
	\end{example}
	
	\begin{proposition}\label{prop:central-equals-corank}
		The central Pl\"ucker vector $\eta^J_\bullet$ equals the planar basis element $\h_J$ modulo $\Lkn$.
	\end{proposition}
	
	\begin{proof}
		Both $\eta^J_\bullet$ and $\h_J$ are determined modulo $\Lkn$ by their tropical cross-ratios.  One checks that 
		\begin{equation*}u^t_I\left(\sum_{M \in \binom{\lbrack n\rbrack}{k}}\eta^J(e_M)e^M\right) = \delta_{I,J}. \hfill \qedhere \end{equation*}
	\end{proof}
	
	For a general positive tropical Pl\"ucker vector $\tropP_\bullet$, we define the \emph{central representative}
	\begin{equation}\label{eq:central-pi}
		\hp_\bullet :=
		\sum_{J \in \bncyc} u_J^t(\tropP_\bullet)\, \eta^{J}_\bullet.
	\end{equation}
	By Proposition~\ref{prop:central-equals-corank}, $\hp_\bullet \equiv \tropP_\bullet$.  We shall view $\hp_\bullet$ both as a tropical Pl\"ucker vector in $\R^{\binom{[n]}{k}}$ and as a piecewise-linear function $\hp(x): \R^n \to \R$.

	\subsection{The gradient of the roof function}
	\label{subsec:gradient}

	On the relative interior of each maximal cell $C$ of $\Delta(\tropP_\bullet)$, the gradient of $\hp(x)$ is a constant vector $v_C := \nabla \hp|_C \in \R^n / \one$.  This defines a piecewise-constant map
	\begin{equation}\label{eq:gradient-map}
		\nabla \hp(x) : \operatorname{int}(\Delta(k,n)) \longrightarrow \R^n / \one,\ \ x \;\mapsto\; v_C \text{ for } x \in \operatorname{relint}(C),
	\end{equation}
	defined on the complement of the codimension-$1$ walls.  The image of this map is the finite set of \emph{vertices} of $B(\hp_\bullet)$.	
	\begin{example}\label{ex:gradient-2block}
		Continuing Example~\ref{ex:central-2block}, for decoration $(2,1)$ on blocks $(123, 456)$ in $(3,6)$, we have with $x_S:= \sum_{s\in S} x_s$,
		\begin{itemize}
			\item Left cell ($x_{123} < 2$): $\hp(x) = -\frac{1}{3}(x_{456} + 3 x_{123})
			= -\frac{1}{3}x_{456} - x_{123}$, 
			\item Right cell ($x_{123} > 2$): $\hp(x) = -\frac{1}{3}(2 x_{123} + 3 x_{456})
			= -\frac{2}{3}x_{123} - x_{456}$,		\end{itemize}
		so 
		\begin{align*}
		v_1 &= (-1, -1, -1, -\frac{1}{3},
			-\frac{1}{3}, -\frac{1}{3}) \sim (\frac{1}{6},\frac{1}{6},\frac{1}{6},\frac{5}{6},\frac{5}{6},\frac{5}{6}) \in \R^6/\one,  \qquad \text{and} \\
			v_2 &= (-\frac{2}{3}, -\frac{2}{3},
			-\frac{2}{3}, -1, -1, -1) \sim (\frac{2}{3},\frac{2}{3},\frac{2}{3},\frac{1}{3},\frac{1}{3},\frac{1}{3}) \in \R^6/\one,
		\end{align*}
		are the two vertices of $B(\hp_\bullet)$.  
		
		For decoration $(1,2)$, the crease moves to $x_{123} = 1$, and the vertices become 
		$$v_1 = (-\frac{1}{3}, -\frac{1}{3}, -\frac{1}{3}, -1, -1,-1)\ \text{and}\ v_2 = (-1, -1, -1, -\frac{2}{3}, -\frac{2}{3}, -\frac{2}{3}).$$
		The two bounded edges are at different positions in $\R^6/\one$.
	\end{example}

	\subsection{The subdifferential and the bounded complex}
	\label{subsec:subdifferential}
	
	The gradient map~\eqref{eq:gradient-map} is undefined at the walls of $\Delta(\tropP_\bullet)$.  The \emph{subdifferential}, a construction common in convex geometry, extends it to a set-valued map defined everywhere; see \cite{Rockafellar} for details.
	\begin{definition}\label{def:subdifferential}
		For a convex function $\kappa : \Delta(k,n) \to \R$, the subdifferential at
		$x \in \Delta(k,n)$ is
		$$
		\widetilde\partial \kappa(x) = \bigl\{ w \in \R^n/\one :	\kappa(y) - \kappa(x) \ge \langle w, y - x \rangle\  \text{for all } y \in \Delta(k,n) \bigr\}.
		$$
	\end{definition}
	
	In our context $\kappa$ is piecewise-linear and convex; in this case the subdifferential interpolates between the gradient values. Indeed, at a generic point of a maximal cell $C$, $\widetilde\partial\hp(x) = \{v_C\}$, a single point (the gradient).  At a codimension-$1$ wall between cells $C_1$ and $C_2$, we have $\widetilde\partial\hp(x) = \conv(v_{C_1}, v_{C_2})$, the segment joining the two gradient values.  More generally, at a face where cells $C_1, \ldots, C_{m+1}$ meet, we have $\widetilde\partial\hp(x) = \conv(v_{C_1}, \ldots, v_{C_{m+1}})$.

	\begin{figure}[h!]
		\centering
		\includegraphics[width=0.7\linewidth]{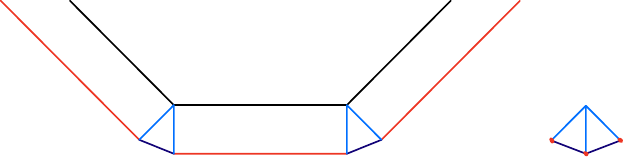}
		\caption{Left: above, the black line is the graph of a piecewise-linear function; below, the (red) tip of the (light blue) gradient vectors and the (dark blue) convex hulls across bend loci.  The subdifferential (right) interpolates between the generic values of the gradient to construct the tropical linear space, which in this case consists of two (dark blue) line segments and three (red) vertices.}
		\label{fig:subdifferential}
	\end{figure}
	
\begin{proposition}\label{prop:bounded-subdiff}
	The geometric realization of the bounded complex is the image of the subdifferential restricted to the interior of $\Delta(k,n)$,
	\begin{equation}\label{eq:bounded-subdiff}
		B(\tropP_\bullet) =
		\bigcup_{x \in \operatorname{int}(\Delta(k,n))}
		\widetilde\partial\hp(x).
	\end{equation}
\end{proposition}

\begin{proof}
	By definition, $\hp: \R^n \to \R$ is the convex piecewise-linear function inducing the positroid subdivision $\Delta(\hp_\bullet) = \Delta(\tropP_\bullet)$. Thus, it coincides on $\Delta(k,n)$ with the lower convex envelope of the height function evaluated at the vertices, $e_I \mapsto \hp_I$.
	
	Let $w \in \R^n/\one$. By Definition \ref{def:subdifferential} (which aligns with the standard definition of the subdifferential in convex analysis, e.g., \cite[Section~23]{Rockafellar}), $w \in \widetilde\partial\hp(x)$ if and only if for all $y \in \Delta(k,n)$,
	\[ 
	\hp(y) - \hp(x) \ge \langle w, y - x \rangle. 
	\]
	Rearranging this inequality, $w \in \widetilde\partial\hp(x)$ if and only if $x$ achieves the global minimum of the function $y \mapsto \hp(y) - \langle w, y \rangle$ over $y \in \Delta(k,n)$.
	
	To characterize these minimizers, we evaluate the objective function using the convex envelope property. For any $y \in \Delta(k,n)$, $\hp(y)$ can be expressed as the minimum over all valid convex combinations:
	\[
	\hp(y) = \min \left\{ \sum_{I \in \binom{[n]}{k}} \lambda_I \hp_I : \sum_{I \in \binom{[n]}{k}} \lambda_I e_I = y, \;\; \lambda_I \ge 0, \;\; \sum_{I \in \binom{[n]}{k}} \lambda_I = 1 \right\}.
	\]
	Substituting this into our objective function and using the linearity of the inner product yields:
	\[
	\hp(y) - \langle w, y \rangle = \min_{\lambda} \sum_{I \in \binom{[n]}{k}} \lambda_I \big( \hp_I - \langle w, e_I \rangle \big).
	\]
	Recall from Section \ref{ssec:tropGr} that the lineality action of $w$ on $\hp_\bullet$ is given by $\hp^w_I = \hp_I - \sum_{i \in I} w_i = \hp_I - \langle w, e_I \rangle$. Let $\alpha(w) := \min_I \hp^w_I$. Because each term in the sum satisfies $\hp^w_I \ge \alpha(w)$, we obtain the universal lower bound:
	\[
	\hp(y) - \langle w, y \rangle \ge \sum_{I \in \binom{[n]}{k}} \lambda_I \alpha(w) = \alpha(w).
	\]
	This minimum value $\alpha(w)$ is achieved if and only if the convex combination strictly assigns positive weight only to vertices $e_I$ for which $\hp^w_I = \alpha(w)$. By definition, the set of indices achieving this minimum forms the bases of the matroid $M(\hp^w) = M(\tropP^w_\bullet)$ (since $\hp_\bullet \equiv \tropP_\bullet \pmod{\Lkn}$). 
	
	Therefore, $x$ is a global minimizer if and only if it can be written as a convex combination of vertices $\{e_I \mid I \in M(\tropP^w_\bullet)\}$. This establishes a fundamental geometric equivalence connecting the subdifferential directly to the matroid polytope:
	\begin{equation} \label{eq:subgrad_equiv}
		w \in \widetilde\partial\hp(x) \iff x \in P_{M(\tropP^w_\bullet)}.
	\end{equation}
	
	We now evaluate the union over the interior of the hypersimplex. A vector $w$ belongs to $\bigcup_{x \in \operatorname{int}(\Delta(k,n))} \widetilde\partial \hp(x)$ if and only if there exists some $x \in \operatorname{int}(\Delta(k,n))$ such that $x \in P_{M(\tropP^w_\bullet)}$. This holds if and only if the matroid polytope $P_{M(\tropP^w_\bullet)}$ intersects the relative interior of $\Delta(k,n)$.
	
	The relative interior of $\Delta(k,n)$ is defined by the strict inequalities $0 < x_i < 1$ for all $i \in [n]$. A closed sub-polytope of $\Delta(k,n)$ intersects the interior if and only if it is not contained entirely in any boundary facet $\{x_i = 0\}$ or $\{x_i = 1\}$. Evaluating this for the matroid polytope $P_{M(\tropP^w_\bullet)}$:
	\begin{itemize}
		\item $P_{M(\tropP^w_\bullet)} \subseteq \{x_i = 0\}$ if and only if $i \notin I$ for all $I \in M(\tropP^w_\bullet)$, which means $i$ is a loop of $M(\tropP^w_\bullet)$.
		\item $P_{M(\tropP^w_\bullet)} \subseteq \{x_i = 1\}$ if and only if $i \in I$ for all $I \in M(\tropP^w_\bullet)$, which means $i$ is a coloop of $M(\tropP^w_\bullet)$.
	\end{itemize}
	Indeed, if $M(\tropP^w_\bullet)$ is loopless and coloopless, its barycenter $y^* = \frac{1}{|M(\tropP^w_\bullet)|} \sum_{I \in M(\tropP^w_\bullet)} e_I$ has coordinates strictly between $0$ and $1$, so $y^* \in \operatorname{int}(\Delta(k,n))$. Therefore, $P_{M(\tropP^w_\bullet)}$ intersects $\operatorname{int}(\Delta(k,n))$ if and only if the matroid $M(\tropP^w_\bullet)$ is both loopless and coloopless. 
	
	By Proposition \ref{prop:Spe}, a vector $w$ lies in the tropical linear space $L(\tropP_\bullet)$ if and only if $M(\tropP^w_\bullet)$ is loopless. Moreover, $w$ belongs to a bounded face of the matroid complex structure of $L(\tropP_\bullet)$ if and only if $M(\tropP^w_\bullet)$ is both loopless and coloopless. Thus, $M(\tropP^w_\bullet)$ being both loopless and coloopless is precisely the condition that $w$ belongs to the bounded complex $B(\tropP_\bullet)$. 
	
	Consequently, $w \in \bigcup_{x \in \operatorname{int}(\Delta(k,n))} \widetilde\partial \hp(x)$ if and only if $w \in B(\tropP_\bullet)$, completing the proof.
\end{proof}
	\begin{example}\label{example: tree 25}
		Figure \ref{fig:tree-25} depicts the bounded complex of $L(\tropP_\bullet)$ in the case $\tropP = \mathfrak{h}_{(12_1 345_1)} + \mathfrak{h}_{(123_1 45_1)}$, which is in the relative interior of a maximal cone of $\Trop_{>0}X(2,5)$.
	\end{example}
	\begin{figure}[ht]
		\centering
		\begin{tikzpicture}[scale=1.2]
			
			\coordinate (L) at (0,0);
			\coordinate (M) at (3,0);
			\coordinate (R) at (6,0);
			
			\draw[very thick] (L) -- (M) node[midway,above] {$e_{12}$};
			\draw[very thick] (M) -- (R) node[midway,above] {$e_{123}$};
			
			\fill (L) circle (2.5pt) node[below=4pt] {$v_A$};
			\fill (M) circle (2.5pt) node[below=4pt] {$v_B$};
			\fill (R) circle (2.5pt) node[below=4pt] {$v_C$};
			
			\draw[dashed,->] (L) -- (-1.2,0.8);
			\draw[dashed,->] (L) -- (-1.2,-0.8);
			\draw[dashed,->] (M) -- (3,1.5);
			\draw[dashed,->] (R) -- (7.2,0.8);
			\draw[dashed,->] (R) -- (7.2,-0.8);
			
			\node[gray,below] at (1.5,-0.5) {\small partition $(12 \mid 345)$};
			\node[gray,below] at (4.5,-0.5) {\small partition $(123 \mid 45)$};
			
		\end{tikzpicture}
		\caption{The bounded complex of $L(\tropP_\bullet)$ for	$\tropP = \h_{(12_1,\, 345_1)} + \h_{(123_1,\, 45_1)}$ in $\Trop_{>0}\Gr(2,5)$: a tree with three vertices and two edges.  The edge directions $e_{12}$ and $e_{123}$ are indicator vectors of cyclic intervals, and the unbounded rays extend from the boundary vertices.}
		\label{fig:tree-25}
	\end{figure}
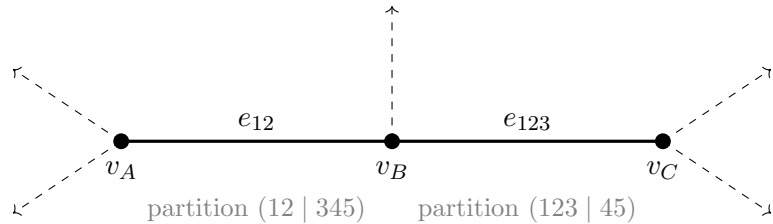
	\begin{example}\label{example: 3-split Gradient}
	Definition \ref{def:central-roof} gives
	$$\eta_{(12_1 34_1 56_1)}(x) = -\frac{1}{3}\min\left(x_{12}+2x_{34} + 3x_{56},x_{34}+2x_{56} + 3x_{12},x_{56}+2x_{12} + 3x_{34}\right).$$
	Let $\eta_\bullet$ denote the corresponding central representative.  The matroid subdivision $\Delta(\eta_\bullet)$ has three maximal faces, respectively
	$$\left\{x\in \Delta_{3,6}: x_{12}\ge 1,\ x_{1234} \ge 2\right\},\ \left\{x\in \Delta_{3,6}: x_{34}\ge 1,\ x_{3456} \ge 2\right\},\ \left\{x\in \Delta_{3,6}: x_{56}\ge 1,\ x_{5612} \ge 2\right\}.$$
	The three vertices $V(B(\tropP_\bullet)) = \{v_1,v_2,v_3\}$ of the bounded complex $B(\eta_\bullet)$ are the values of the gradients of $\eta$ over the interiors of these three maximal faces:
	\begin{eqnarray*}
		v_1 = \left(-\frac{1}{3},-\frac{1}{3},-\frac{2}{3},-\frac{2}{3},-1,-1\right),\ v_2 = \left(-1,-1,-\frac{1}{3},-\frac{1}{3},-\frac{2}{3},-\frac{2}{3}\right),\ v_3 = \left(-\frac{2}{3},-\frac{2}{3},-1,-1,-\frac{1}{3},-\frac{1}{3}\right).
	\end{eqnarray*}
	\begin{figure}[h!]
		\centering
		\includegraphics[width=0.30\linewidth]{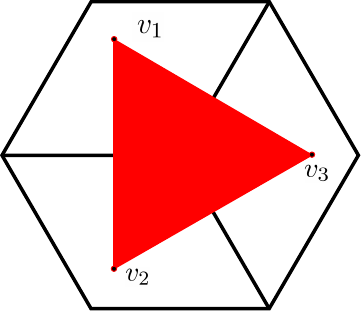}
		\caption{The subdivided black hexagon is the image of the 3-split subdivision of $\Delta_{3,6}$ under the projection $x\mapsto (x_1+x_2,x_3+x_4,x_5+x_6)$.  The red triangle is the image of $\Delta_{3,6}$ under the subdifferential, i.e. by Proposition \ref{prop:bounded-subdiff} this is the bounded complex of the tropical linear space.}
		\label{fig:3splithexagon}
	\end{figure}
	
\end{example}
	
	\subsection{The roof function identity}
	\label{subsec:roof-identity}
	
	The planar basis expansion $\tropP_\bullet \equiv \sum_{J \in \bncyc} u_J^t(\tropP_\bullet)\, \h_J \pmod{\Lkn}$ lifts to an identity of central roof functions on~$\Delta(k,n)$.
	
	\begin{lemma}\label{lem:roof-identity}
		Let $\tropP_\bullet \in \Trop_{>0}\Gr(k,n)$ have central representative $\hp_\bullet = \sum_{J \in \bncyc} u_J^t(\tropP_\bullet)\, \eta^J_\bullet$ (see  \eqref{eq:central-pi}).  Then
		\begin{equation}\label{eq:roof-identity}
			\hp(x) \;=\;
			\sum_{J \in \bncyc} u_J^t(\tropP_\bullet)\, \eta^J(x)
		\end{equation}
		for all $x \in \Delta(k,n)$.
	\end{lemma}
	
	\begin{proof}
		Both sides are piecewise-linear functions on $\Delta(k,n)$ that agree at every vertex $e_I$, since $\hp_I = \sum_J u_J^t(\tropP_\bullet)\, \eta^J_I$ by definition.  Each $\eta^J$ has creases along hyperplanes $\sum_{i \in A} x_i = r$ for cyclic intervals $A$ and integers~$r$; these are the walls of the alcoved triangulation of~$\Delta(k,n)$. Therefore the right-hand side is piecewise-linear on the alcoved triangulation, regardless of the signs of the coefficients $u_J^t(\tropP_\bullet)$.  On the other hand, since $\tropP_\bullet$ is positive, $\Delta(\tropP_\bullet) = \Delta(\hp_\bullet)$ is a positroid subdivision, and thus a coarsening of the alcoved triangulation.  Since both sides are linear on alcoves and agree at all vertices, they are equal as functions.	\end{proof}

\section{Diameter of a Tropical Linear Space}
\label{sec:global-bounds}

In this section, we use the tools we have developed thus far to compute the diameter of a positive tropical linear space. We show that the PK weight provides an upper bound on the diameter of the bounded complex of a positive tropical linear space, embedding it inside the fundamental alcoved simplex dilated by $\wt_\PK(\tropP_\bullet)$.

Because the tropical linear space $L(\tropP_\bullet)$ is invariant under the lineality action up to translation in $\R^n/\one$, its position and the bounding alcoved simplex depend on the lineality representative of the Pl\"ucker vector. The tightest bounding into the untranslated fundamental alcoved simplex $\Delta_{\mathrm{std}}$ requires fixing the translation properly, which we accomplish by anchoring the space with a \emph{balanced} central representative $\hp_\bullet$.

\begin{definition}\label{def:std-simplex}
	The \emph{fundamental alcoved simplex} $\Delta_{\mathrm{std}} \subset \R^n/\one$ is the polytope
	\[
	\Delta_{\mathrm{std}}
	= \bigl\{\x \in \R^n/\one : x_u - x_v \le 1 \text{ for all } u, v \in [n]\bigr\}.
	\]
	For $C \ge 0$, the dilated simplex $C \cdot \Delta_{\mathrm{std}}$ is the region $\max_{u,v}(x_u - x_v) \le C$.
\end{definition}

To establish the bound, we break the problem into two steps. First, we identify a uniform non-negativity property for the balanced central representative. Second, we apply the Grassmann necklace to the shifted Pl\"ucker vector to bound the coordinates of any point in the complex directly from the positive tropical Pl\"ucker relations.

\subsection{The balanced central representative}

For any representative $\tropP_\bullet$, the sum of the cyclic-gap differences around the cycle is independent of lineality. Specifically, the total sum evaluates to the bridge function:
\[
\sum_{j=0}^{n-1} \bigl(\tropP_{\mathrm{cyc}(j)} - \tropP_{\mathrm{gap}(j)}\bigr) = H(\tropP_\bullet) = \wt_{\PK}(\tropP_\bullet).
\]
Because the PK weight of a positive tropical Pl\"ucker vector is non-negative, the sum of the differences is non-negative. 

\begin{lemma}\label{lem:balanced-central}
	Modulo lineality, every positive tropical linear space $L(\tropP_\bullet)$ admits a \emph{balanced central representative} $\hp_\bullet$ such that the cyclic-gap differences are uniformly distributed and strictly non-negative:
	\[
	\hp_{\mathrm{cyc}(j)} - \hp_{\mathrm{gap}(j)} = \frac{1}{n} \wt_{\PK}(\tropP_\bullet) \;\ge\; 0
	\]
	for all $j = 0, \ldots, n-1$.
\end{lemma}
\begin{proof}
	We seek a lineality shift $\y \in \R^n$ such that the shifted representative $\hp = \tropP^\y$ satisfies the condition. By definition, $\hp_I = \tropP_I - \sum_{i \in I} y_i$. The cyclic-gap difference transforms as:
	\[
	\hp_{\mathrm{cyc}(j)} - \hp_{\mathrm{gap}(j)} = \tropP_{\mathrm{cyc}(j)} - \tropP_{\mathrm{gap}(j)} - y_{j+k} + y_{j+k+1}.
	\]
	Setting this equal to $\frac{1}{n} \wt_{\PK}(\tropP_\bullet)$ yields a system of cyclic difference equations:
	\[
	y_{j+k} - y_{j+k+1} = \tropP_{\mathrm{cyc}(j)} - \tropP_{\mathrm{gap}(j)} - \frac{1}{n} \wt_{\PK}(\tropP_\bullet).
	\]
	A cyclic system $y_m - y_{m+1} = c_m$ has a solution if and only if $\sum c_m = 0$. Here, the sum of the right-hand side over all $j$ is $\wt_{\PK}(\tropP_\bullet) - n \cdot \frac{1}{n} \wt_{\PK}(\tropP_\bullet) = 0$. Thus, a valid shift $\y$ always exists, making the balanced central representative well-defined.
\end{proof}

\subsection{Local coordinate bounds via the Grassmann necklace}

The second nonnegativity property is structural and applies to the shifted Pl\"ucker vector representing any bounded face of the tropical linear space.

\begin{lemma}\label{lem:cyc-gap-shifted-nonneg}
	Let $\nu_\bullet \in \R^{\binom{[n]}{k}}$ be a positive tropical Pl\"ucker vector such that the matroid $M(\nu_\bullet) = \{I \in \binom{[n]}{k} : \nu_I = \min_J \nu_J\}$ is a loopless and coloopless positroid. Then $\nu_{\mathrm{cyc}(j)} \ge \nu_{\mathrm{gap}(j)}$ for all $j=0, 1, \dots, n-1$.
\end{lemma}

\begin{proof}
	By cyclic symmetry, we may assume $j=0$. Let $C = \mathrm{cyc}(0) = \{1, \dots, k\}$ and $G = \mathrm{gap}(0) = \{1, \dots, k-1, k+1\}$. We want to show $\nu_C \ge \nu_G$.
	
	Let $\alpha = \min_J \nu_J$. Consider the Grassmann necklace $(I_1, \dots, I_n)$ of the positroid $M(\nu)$. By definition, $I_{k+1}$ is the lexicographically minimal basis of $M(\nu)$ with respect to the cyclically shifted order $k+1 < k+2 < \dots < n < 1 < \dots < k$. 
	
	Because $M(\nu_\bullet)$ is loopless, $k+1$ is not a loop, meaning it is contained in at least one basis. The greedy algorithm constructing $I_{k+1}$ considers $k+1$ first and will therefore select it, so $k+1 \in I_{k+1}$. Because $M(\nu_\bullet)$ is coloopless, $k$ is not a coloop, meaning the set of all other elements $[n] \setminus \{k\}$ spans the matroid. Since $k$ is the absolute last element in this cyclic order, the greedy algorithm will have already completed a basis before reaching $k$. Thus, $k \notin I_{k+1}$.
	
	Let $D = I_{k+1} \setminus \{k+1\}$. We have $|D| = k-1$ and $D \cap \{k, k+1\} = \emptyset$. Since $I_{k+1} \in M(\nu)$, we have $\nu_{D \cup \{k+1\}} = \alpha$. Because $\alpha$ is the global minimum over all $k$-subsets, we trivially have $\nu_{D \cup \{k\}} \ge \alpha$. Consequently:
	\begin{equation}\label{eq:grassmann_diff}
		\nu_{D \cup \{k\}} - \nu_{D \cup \{k+1\}} \ge 0.
	\end{equation}
	
	We will now walk from $X = \{1, \dots, k-1\}$ to $D$ by exchanging elements, using the $3$-term positive tropical Pl\"ucker relation to show that the difference between the $k$ and $k+1$ evaluations monotonically bounds the starting difference. 
	
	Let $X \setminus D = \{x_1, \dots, x_m\}$ and $D \setminus X = \{y_1, \dots, y_m\}$. Note that every $x_i \in X = \{1, \dots, k-1\}$ and every $y_i \in D \setminus X \subseteq \{k+2, \dots, n\}$.
	
	For $i=0, \dots, m$, define the intermediate sets:
	\[
	X_i = (X \cap D) \cup \{y_1, \dots, y_i, x_{i+1}, \dots, x_m\},
	\]
	so that $X_0 = X$ and $X_m = D$. For $i=1, \dots, m$, define the $(k-2)$-subset $S_i = X_i \setminus \{y_i\} = X_{i-1} \setminus \{x_i\}$. 
	
	Notice that the four elements $x_i, k, k+1, y_i$ are disjoint from $S_i$ and appear in exactly that cyclic order ($1 \le x_i < k < k+1 < y_i \le n$). We apply the positive tropical Pl\"ucker relation \eqref{eq:posTropPlucker} on $S_i$ and these four elements:
	\[
	\nu_{S_i \cup \{x_i, k+1\}} + \nu_{S_i \cup \{k, y_i\}} = \min\bigl(\nu_{S_i \cup \{x_i, k\}} + \nu_{S_i \cup \{k+1, y_i\}},\ \nu_{S_i \cup \{x_i, y_i\}} + \nu_{S_i \cup \{k, k+1\}} \bigr).
	\]
	Because the left-hand side is the minimum, it is bounded above by the first argument:
	\[
	\nu_{S_i \cup \{x_i, k+1\}} + \nu_{S_i \cup \{k, y_i\}} \le \nu_{S_i \cup \{x_i, k\}} + \nu_{S_i \cup \{k+1, y_i\}}.
	\]
	Substituting $X_{i-1} = S_i \cup \{x_i\}$ and $X_i = S_i \cup \{y_i\}$, this inequality evaluates to:
	\[
	\nu_{X_{i-1} \cup \{k+1\}} + \nu_{X_i \cup \{k\}} \le \nu_{X_{i-1} \cup \{k\}} + \nu_{X_i \cup \{k+1\}}.
	\]
	Rearranging yields:
	\[
	\nu_{X_{i-1} \cup \{k\}} - \nu_{X_{i-1} \cup \{k+1\}} \ge \nu_{X_i \cup \{k\}} - \nu_{X_i \cup \{k+1\}}.
	\]
	Summing this inequality from $i=1$ to $m$ creates a telescoping sum:
	\[
	\nu_{X_0 \cup \{k\}} - \nu_{X_0 \cup \{k+1\}} \ge \nu_{X_m \cup \{k\}} - \nu_{X_m \cup \{k+1\}}.
	\]
	Since $X_0 = X$ and $X_m = D$, this evaluates exactly to:
	\[
	\nu_C - \nu_G \ge \nu_{D \cup \{k\}} - \nu_{D \cup \{k+1\}}.
	\]
	Applying \eqref{eq:grassmann_diff}, we conclude $\nu_C - \nu_G \ge 0$. Thus, $\nu_{\mathrm{cyc}(0)} \ge \nu_{\mathrm{gap}(0)}$, as desired.
\end{proof}

\subsection{Global diameter bounds}

With these two properties, the diameter bound for the central tropical linear space follows naturally by telescoping around the cycle.

\begin{theorem}\label{thm:global_bounding_simplex}
	For any $\tropP_\bullet \in \Trop_{>0}\Gr(k,n)$, let $\hp_\bullet$ be its balanced central representative. We have
	\begin{equation}\label{eq:global-bound}
		B(\hp_\bullet) \subseteq \wt_{\PK}(\tropP_\bullet) \cdot \Delta_{\mathrm{std}}.
	\end{equation}
\end{theorem}

\begin{proof}
	Let $\x \in B(\hp_\bullet)$. We show that $x_u - x_v \le \wt_{\PK}(\tropP_\bullet)$ for all $u, v \in [n]$. 
	
	Define the shifted vector $\nu = \hp^{\x}_\bullet$, given by $\nu_I = \hp_I - \sum_{i \in I} x_i$.  By Proposition~\ref{prop:Spe}, since $\x$ belongs to a bounded face of the matroid complex structure of $L(\hp_\bullet)$, the matroid $M(\nu_\bullet)$ is a loopless and coloopless positroid.
	
	By Lemma~\ref{lem:cyc-gap-shifted-nonneg}, $\nu_{\mathrm{cyc}(j)} \ge \nu_{\mathrm{gap}(j)}$ for all $j$. Expanding $\nu$ back into $\hp_\bullet$ and $\x$ gives:
	\[
	\hp_{\mathrm{cyc}(j)} - \sum_{m \in \mathrm{cyc}(j)} x_m \ge \hp_{\mathrm{gap}(j)} - \sum_{m \in \mathrm{gap}(j)} x_m.
	\]
	Since $\mathrm{cyc}(j) \setminus \mathrm{gap}(j) = \{j+k\}$ and $\mathrm{gap}(j) \setminus \mathrm{cyc}(j) = \{j+k+1\}$ (with indices taken cyclically modulo $n$), the shared coordinates cancel to leave exactly the local boundary gradients:
	\begin{equation}\label{eq:cyclic_bounds}
		x_{j+k} - x_{j+k+1} \le \hp_{\mathrm{cyc}(j)} - \hp_{\mathrm{gap}(j)}.
	\end{equation}
	
	For arbitrary $u, v \in [n]$, we can express the difference $x_u - x_v$ as a telescoping sum of consecutive cyclic differences around the circle from $u$ to $v-1$:
	\[
	x_u - x_v = \sum_{m=u}^{v-1} (x_m - x_{m+1}).
	\]
	Applying our local bound \eqref{eq:cyclic_bounds} to each step $m = j+k$, we obtain:
	\[
	x_u - x_v \le \sum_{m=u}^{v-1} \bigl(\hp_{\mathrm{cyc}(m-k)} - \hp_{\mathrm{gap}(m-k)}\bigr).
	\]
	By Lemma~\ref{lem:balanced-central}, every term $\hp_{\mathrm{cyc}} - \hp_{\mathrm{gap}}$ in this sum is strictly non-negative. Therefore, the partial sum along the cycle from $u$ to $v-1$ is strictly bounded above by the total sum around the entire cycle:
	\[
	x_u - x_v \le \sum_{j=0}^{n-1} \bigl(\hp_{\mathrm{cyc}(j)} - \hp_{\mathrm{gap}(j)}\bigr).
	\]
	This cyclic sum evaluates to $\wt_{\PK}(\tropP_\bullet)$. Since $u$ and $v$ were arbitrary, it follows that
	\[
	\max_{u,v} (x_u - x_v) \le \wt_{\PK}(\tropP_\bullet).
	\]
	This system of inequalities defines exactly the dilated fundamental alcoved simplex, proving $\x \in \wt_{\PK}(\tropP_\bullet) \cdot \Delta_{\mathrm{std}}$.
\end{proof}

\begin{example}
	The bounded complex in Example \ref{example: tree 25} has two line segments, and its Pl\"ucker vector $\tropP_\bullet = \h_{25} + \h_{35}$ has PK weight 2; therefore the bounded complex of its central representative $B(\hp_\bullet)$ embeds into the second dilate of the fundamental alcoved simplex $2\Delta_{\mathrm{std}}$. 
\end{example}

\begin{example}
	The bounded complex in Example \ref{example: 3-split Gradient} is a single triangle, and its tropical Pl\"ucker vector has PK weight 1.  It is \emph{equal} to the fundamental alcoved simplex.
\end{example}

\begin{example}\label{example: 312 embedding}
	Consider $\tropP_\bullet \in \Trop_{>0}\Gr(3,12)$ with planar basis expansion
	\[\tropP_\bullet= -\h_{1,4,11}+\h_{1,4,12}-\h_{1,8,10}+\h_{1,8,11}+\h_{1,9,10}-\h_{2,4,7}+\h_{2,4,11}+\h_{2,5,7}+\h_{3,4,7}-\h_{5,7,10}+\h_{5,8,10}+\h_{6,7,10},\]
	computed from the strand triples of the scaffold on the left of Figure \ref{fig:312canvas}.  By Theorem \ref{thm:noncrossingmain}, $\tropP_\bullet$ is, up to lineality, uniquely characterized by its projection
	\[\Psi(\tropP_\bullet) = \vn_{1,9,10}+\vn_{2,5,7}+\vn_{3,4,12}+\vn_{6,8,11},\]
	which, by Theorem \ref{thm:f-pk-equals-nc}, has PK weight four, the number of columns in its noncrossing tableau.  This is also the size of the bounding simplex for the bounded complex of the central tropical linear space, see Figure \ref{fig:312canvas}, and compare with our companion paper \cite{EL}.
\end{example}

\end{document}